\documentclass{siamart171218}

\usepackage{bbm}
\newtheorem{remark}[theorem]{Remark}

\usepackage{mathtools}

\usepackage{enumitem}
\usepackage{lipsum}
\usepackage{amsfonts}
\usepackage{graphicx}
\usepackage{epstopdf}
\usepackage{algorithmic}
\ifpdf
  \DeclareGraphicsExtensions{.eps,.pdf,.png,.jpg}
\else
  \DeclareGraphicsExtensions{.eps}
\fi

\newcommand{\TheTitle}{Optimizing weighted ensemble sampling} 
\newcommand{\TheAuthors}{D. Aristoff and D. Zuckerman}

\headers{\TheTitle}{\TheAuthors}

\title{{Optimizing weighted ensemble sampling of steady states}\thanks{This 
work was supported by the
National Science Foundation 
via the awards
NSF-DMS-1522398 and NSF-DMS-1818726 and the National Institutes of Health via the award GM115805.}}

\author{
  David Aristoff\thanks{Colorado State University
    (\email{aristoff@math.colostate.edu}) }
  \and
  Daniel M.  Zuckerman\thanks{Oregon Health \& Science University (\email{zuckermd@ohsu.edu})}
}

\usepackage{amsopn}

\begin{document}

\maketitle

\begin{abstract}
We propose parameter optimization techniques for weighted ensemble sampling 
of Markov chains in the steady-state regime. Weighted ensemble 
consists of replicas of a Markov chain, each carrying a weight, 
that are periodically resampled according to their 
weights inside of each of a number of {bins} that partition state space. We derive, from first principles, strategies for optimizing the choices of weighted ensemble parameters, in particular the choice of 
bins and the number of replicas to maintain in each bin. In a simple numerical example, we 
compare our new strategies with more traditional ones
and with direct Monte Carlo.
\end{abstract}



\maketitle
\section{Introduction}

Weighted ensemble is a Monte Carlo
method 
based on stratification and resampling, 
originally designed to solve 
problems in computational 
chemistry~\cite{bhatt2010steady,
chong2017path,jeremy,costaouec2013analysis,darve2013computing,
donovan2013efficient,huber1996weighted,
rojnuckarin1998brownian,
rojnuckarin2000bimolecular,suarez,
zhang2007efficient,
zhang2010weighted,zuckerman,zwier2015westpa}. 
Weighted ensemble currently has a substantial 
user base; see~\cite{westpa} for 
software and a more complete list of 
publications. In general terms, the method consists of periodically 
resampling from an ensemble of weighted replicas 
of a Markov process. In each of a number 
of {\em bins}, a certain number of replicas 
is maintained according to a prescribed replica {\em allocation}. 
The weights are adjusted so that the weighted 
ensemble has the correct distribution~\cite{zhang2010weighted}. 
In the context of rare event or 
small probability calculations, weighted 
ensemble can significantly 
outperform 
direct Monte Carlo, or independent 
replicas; the performance gain comes 
from allocating 
more particles to important or rare 
regions of state space. In this 
sense, weighted ensemble can be 
understood as an importance sampling 
method. We assume the underlying 
Markov process is expensive to simulate, 
so that optimizing variance versus 
cost is critical. In applications, 
the Markov process is usually a high dimensional drift diffusion,
such as Langevin molecular dynamics~\cite{stoltz2010free}, 
or a continuous time Markov chain 
representing reaction network 
dynamics~\cite{Kurtz}.

We focus on the computation of the 
average of a given function or
{\em observable} with respect to the 
unique steady state of the 
Markov process, though many of 
our ideas could also be applied 
in a finite time setting. 
One of the most common applications 
of weighted ensemble is the computation of a
{\em mean first passage time}, or the
mean time for a Markov process to go from 
an initial state to some target 
set.
The mean first passage time is an important quantity in physical and chemical processes, but it can be prohibitively large to compute using direct Monte Carlo simulation~\cite{adhikari2019computational,hofmann2003mean,metzler2014first}.
In weighted ensemble, the mean 
first passage time to a target set is reformulated, 
via the Hill relation~\cite{hill2005free}, 
as the inverse of the steady 
state flux into the target. Here, 
the observable is the characteristic or 
indicator
function of the target set, and 
the flux is the steady-state 
average of this observable.  This 
steady state can sometimes be accessed on 
time scales much smaller 
than the mean first passage 
time~\cite{adhikari2019computational,jeremy2,danonline}. 
On the other hand, 
when the mean first passage time is 
very large, 
the corresponding steady state flux is very small and needs to be 
estimated with substantial precision, 
which is why importance sampling is needed. 
The steady state in this case, and in 
most weighted ensemble applications, 
is {\em not} known up to a normalization 
factor. As a result, many standard Monte Carlo methods, 
including those based on Metropolis-Hastings, do not apply.

In this article
we consider weighted ensemble 
parameter optimization for 
steady state averages. Our work here 
expands on related results in~\cite{aristoff2016analysis} 
for weighted ensemble 
optimization on finite time horizons.
For a given Markov chain, weighted 
ensemble is completely characterized 
by the choice of resampling times, 
bins, and replica allocation.
{\em In this article we discuss how 
to choose the bins and replica allocation}. 
We also argue that the frequency of resampling times 
is mainly limited by processor 
interaction cost and not variance.
We use a first-principles, finite 
replica analysis  
based on Doob decomposing the 
weighted ensemble
variance. Our earlier work~\cite{aristoff2016analysis} 
is based on this same mathematical technique, but 
the methods described there  
are not suitable for long-time computations or  
bin optimization~\cite{aristoff2019ergodic}.
As is usual in importance sampling, 
our parameter optimizations 
require estimates of the very
variance terms we want to minimize. 
However, because {\em weighted ensemble is 
statistically exact no matter the parameters}~\cite{aristoff2019ergodic,zhang2010weighted}, 
these estimates can be crude. 
The choice of parameters affects 
the variance, not the mean;  
we only require parameters that 
are good enough to  
beat direct Monte Carlo.

From the point of view of applications, similar particle methods employing stratification 
include Exact Milestoning~\cite{bello2015exact}, 
Non-Equilibrium Umbrella 
Sampling~\cite{warmflash2007umbrella,dinner2016trajectory},  Transition 
Interface Sampling~\cite{van2003novel},  
Trajectory Tilting~\cite{vanden2009exact},
and Boxed Molecular dynamics~\cite{glowacki2011boxed}.  
There are related methods based on 
sampling {reactive paths},  
or paths going directly from a 
source to a target, in the context 
of the mean first passage 
time 
problem just cited. Such methods 
include Forward Flux Sampling~\cite{allen2006forward} 
and Adaptive Multilevel Splitting~\cite{brehier_AMS,brehier_AMS2,tony_AMS}. 
These methods differ from weighted 
ensemble in that they estimate the mean 
first passage time directly from reactive paths 
rather than from steady state and 
the Hill relation. In 
contrast with many of these methods, weighted ensemble 
is simple enough to allow for a 
relatively straightforward non-asymptotic variance 
analysis based on Doob decomposition~\cite{doobbook}.

From a mathematical viewpoint,
weighted ensemble is simply a resampling-based 
evolutionary algorithm. In this 
sense it resembles particle filters, sequential 
importance sampling, and  
sequential Monte Carlo. 
For a review 
of sequential Monte Carlo, see the 
textbook~\cite{del2004feynman}, the articles~\cite{del2014particle,del2005genealogical} or the 
compilation~\cite{doucet2001sequential}. 
There is some recent work on optimizing 
the Gibbs-Boltzmann input potential functions in sequential
Monte Carlo~\cite{balesdent2013optimisation,
chraibi2018optimal,jacquemart2016tuning,webber1,wouters2016rare} 
(see also~\cite{del2005genealogical}). 
We emphasize that weighted ensemble 
is different from most sequential Monte Carlo methods, 
as it relies on a 
bin-based resampling 
mechanism rather than a globally defined 
fitness function like a Gibbs-Boltzmann potential. In particular, sequential 
Monte Carlo is more commonly used 
to sample rare events on finite time horizons, and may not be appropriate 
for very long time computations of 
the sort considered here, as explained in~\cite{aristoff2019ergodic}. 

To our knowledge, the binning and 
particle allocation strategies we derive here are 
new. A similar allocation strategy 
for weighted ensemble on finite time 
horizons was proposed in~\cite{aristoff2016analysis}. Our allocation strategy, which minimizes {\em mutation variance} -- the 
variance corresponding 
to evolution of the replicas --
extends
ideas from~\cite{aristoff2016analysis} to 
our steady-state setup. 
We draw an analogy between minimizing 
mutation variance and minimizing sensitivity 
in an Appendix. In most weighted ensemble 
simulations, and in~\cite{aristoff2016analysis}, 
bins are chosen in an ad hoc way. We 
show below, however, that choosing 
bins carefully is important, particularly if 
relatively few can be afforded.
We propose a new binning strategy based on minimizing 
{\em selection variance} -- the 
variance associated with 
resampling from the replicas -- in which 
weighted ensemble bins are 
aggregated from a collection 
of smaller {\em microbins}~\cite{jeremy}.  

Formulas for the mutation 
and selection variance are 
derived in a
companion paper~\cite{aristoff2019ergodic} 
that proves an ergodic theorem 
for weighted ensemble time averages.
These variance formulas, which we use 
to define our allocation and bin 
optimizations, 
involve the Markov kernel $K$ that describes 
the evolution of the underlying Markov process between 
resampling times, as well as the solution, $h$, of a certain Poisson equation. 
We propose estimating $K$ 
and $h$ with 
techniques from Markov 
state modeling~\cite{husic2018markov,pande2010everything,
sarich2010approximation,
schutte2013metastability}. In this formulation, 
the microbins correspond 
to the Markov states, and 
a microbin-to-microbin transition 
matrix  
defines the approximations of 
$K$ and $h$.

This article is organized as follows. 
In Section~\ref{sec:alg}, we 
introduce notation, give 
an overview of  
weighted ensemble, and describe 
the parameters we wish to optimize. 
In 
Section~\ref{sec:notation}, we 
present weighted ensemble in 
precise detail (Algorithm~\ref{alg1}), reproduce the aforementioned
variance 
formulas (Lemmas~\ref{lem_mut_var} and~\ref{lem_sel_var}) from our companion 
paper~\cite{aristoff2019ergodic}, 
and describe the solution, $h$, to 
a certain Poisson equation arising from these formulas (equations~\eqref{h} and~\eqref{h2}).  
In Section~\ref{sec:var_min}, 
we introduce novel optimization 
problems (equations~\eqref{opt_allocation} 
and~\eqref{opt_bins}) for 
choosing the bins and particle allocation. 
The resulting allocation in~\eqref{Ntu} 
can be seen as a version of equation 
(6.4) in~\cite{aristoff2016analysis}, 
modified for steady state and our new 
microbin setup.
These optimizations are idealized 
in the sense that they involve 
$K$ and $h$, which cannot be computed 
exactly. Thus in Section~\ref{sec:params}, we 
propose using microbins and Markov 
state modeling to estimate 
their solutions
(Algorithms~\ref{alg_opt_allocation} and~\ref{alg_opt_bins}). In Section~\ref{sec:numerics}, we test our 
methods with a simple numerical example 
(Figures~\ref{fig_V_h_pi}-\ref{fig_optimization_data}). 
Concluding remarks and suggestions 
for future work are in Section~\ref{sec:remarks}. 
In the Appendix, we describe residual 
resampling~\cite{douc2005comparison}, a common 
resampling technique, and draw an analogy 
between sensitivity and our 
mutation variance minimization strategy.

\section{Algorithm}\label{sec:alg}
A weighted ensemble 
consists of a collection 
of replicas, or {\em particles}, belonging 
to a common state space, with associated 
positive scalar weights. The particles repeatedly 
undergo resampling and evolution 
steps. Using a genealogical analogy, we refer to particles 
before resampling as {\em parents}, 
and just after resampling as {\em children}. A 
child is initially just a copy of 
its parent, though it evolves 
independently of other children 
of the same parent. The total weight 
remains constant in time, which 
is important for the stability of 
long-time averages~\cite{aristoff2019ergodic} (and is a critical difference between 
the optimization algorithm described in~\cite{aristoff2016analysis} and 
the one outlined here). 
Between resampling steps, the particles evolve 
independently via the
same Markov kernel $K$.
The initial parents can be 
arbitrary, though their weights 
must sum to $1$; see Algorithm~\ref{alg_initialize} for a description 
of an initialization 
step tailored to steady-state calculations. 
The number $N_{\textup{init}}$ of initial 
particles can be larger than 
the number $N$ of weighted ensemble 
particles after the first 
resampling step.

For the resampling or {\em selection} step,
we require a collection of {\em bins}, 
denoted ${\mathcal B}$, and 
a {\em particle allocation} $N_t(u)_{t \ge 0}^{u \in {\mathcal B}}$, where $N_t(u)$ is the (nonnegative integer) number 
of children in bin $u \in {\mathcal B}$ 
at time $t$. We will assume the bins 
form a partition of state space (though 
in general they can be any grouping 
of the particles~\cite{aristoff2019ergodic}). 
The total number of 
children is always $N = \sum_{u \in {\mathcal B}} N_t(u)$. {\em Both the bins and 
particle allocation are user-chosen 
parameters, 
and to a large extent this article 
concerns how to pick these parameters.}
In general, the bins and particle 
allocation can be time dependent and adaptive.
For simpler presentation, however,
we assume that
the bins are based on a fixed 
partition of state space. After selection, 
the weight of each child in bin 
$u \in {\mathcal B}$
is $\omega_t(u)/N_t(u)$, 
where $\omega_t(u)$ is the sum of the 
weights of all the parents in bin $u$. By 
construction, this 
preserves the total weight $\sum_{u \in {\mathcal B}} \omega_t(u) = 1$.
In every occupied bin~$u$, the $N_t(u) \ge 1$ children are selected from the parents in bin $u$ with probability 
proportional to the parents' weights. (Here, we mean bin $u$ is occupied if $\omega_t(u)>0$. In
unoccupied bins, where $\omega_t(u) = 0$, we 
set $N_t(u)  = 0$.)

In the evolution or {\em mutation} step, 
the time $t$ advances, and
all the children independently evolve one  
step according to a 
fixed Markov kernel $K$, 
keeping their weights from 
the selection step, and
becoming the next parents. 
In practice, $K$ 
corresponds to 
the underlying 
process evaluated 
at {\em resampling 
times} $\Delta t$. 
That is, $K$ is a $\Delta t$-{\em skeleton} of 
the underlying 
Markov process~\cite{meyn}. 
For the mathematical analysis below, 
this will only 
be important
when we consider 
varying the resampling 
times, and in particular 
the $\Delta t \to 0$ limit. 
(We think of this 
underlying process 
as being continuous 
in time, though 
of course 
time discretization 
is usually required 
for simulations. Weighted 
ensemble is used {\em on top} 
of an integrator of the 
underlying process; 
in particular weighted 
ensemble does not handle 
the time discretization. We will 
not be concerned with 
the unavoidable 
error resulting from time 
discretization.)
{\em We 
assume that 
$K$ is uniformly geometrically
ergodic~\cite{doucbook}} with respect to a stationary
distribution, or steady state, $\mu$. Recall we 
are interested in 
steady-state averages. Thus 
for a given bounded real-valued function or {\em observable} $f$ on state space, we estimate
$\int f\,d\mu$ at each time by 
evaluating the weighted sum of $f$ 
on the current collection of parent particles.

We summarize weighted ensemble as follows:
\vskip2pt
\begin{itemize}[leftmargin=20pt]
\item In the {selection step} 
at time $t$, inside each bin $u$, we resample $N_t(u)$ children 
from the parents in $u$, according to the 
distribution defined by their 
weights. After selection, all the 
children in bin $u$ have the same 
weight $\omega_t(u)/N_t(u)$, 
where $\omega_t(u)$ is the total 
weight in bin $u$ before and after selection.
\item In each {mutation step}, 
the children evolve independently 
according to the Markov kernel $K$. After evolution, these children become the new parents.
\item The weighted ensemble evolves 
by repeated selection and then 
mutation steps. The time $t$ advances after a single pair of selection and mutation steps.
\end{itemize}
\vskip2pt
See Algorithm~\ref{alg1} for a 
detailed description of weighted ensemble. 

An important property of 
weighted ensemble is that 
it is unbiased no matter 
the choice of parameters: at 
time $t$ the weighted particles 
have the 
same distribution 
as a Markov chain 
evolving according to $K$. 
See Theorem~\ref{thm_unbiased}. 
 With
bad parameters, however, 
weighted ensemble can suffer 
from large variance, even 
worse than direct Monte Carlo. 
As there is no free lunch, choosing 
parameters cleverly requires either some  information 
about $K$, perhaps 
gleaned from prior simulations 
or obtained adaptively during 
weighted ensemble simulations. 
We will assume we have a 
collection of {\em microbins} 
which we use to gain information about 
$K$. The microbins, like the weighted 
ensemble bins, form a partition 
of state space, and each bin 
will be a union of microbins. 
We use the term microbins 
because the microbins may be smaller than 
the actual weighted ensemble bins. 
We discuss the reasoning behind this 
distinction in Section~\ref{sec:params}; see Remark~\ref{rmk_bins}.

\vskip5pt

\section{Mathematical notation and algorithm}\label{sec:notation}

We write
$\xi_t^1,\ldots,\xi_t^{N}$ for the 
parents at time $t$ and $\omega_t^1,\ldots,\omega_t^{N}$ for their weights. 
Their children are denoted
$\hat{\xi}_t^1,\ldots,\hat{\xi}_t^{N}$ 
with weights $\hat{\omega}_t^1,\ldots,\hat{\omega}_t^{N}$.
Thus, weighted ensemble advances in time as follows: 
\begin{align*}
&\{\xi_t^i\}^{i=1,\ldots,N}
\xrightarrow{\textup{selection}} 
\{{\hat \xi}_t^i\}^{i=1,\ldots,N}
\xrightarrow{\textup{mutation}} 
\{{\xi}_{t+1}^i\}^{i=1,\ldots,N},\\
&\{\omega_t^i\}^{i=1,\ldots,N}
\xrightarrow{\textup{selection}} 
\{{\hat \omega}_t^i\}^{i=1,\ldots,{N}}
\xrightarrow{\textup{mutation}} 
\{{\omega}_{t+1}^i\}^{i=1,\ldots,N}.
\end{align*}

The particles belong a 
common standard Borel state space~\cite{durrett2019probability}. 
This state space is divided into 
a finite collection ${\mathcal B}$ of 
disjoint subsets
(throughout we only consider 
measurable sets and functions). We 
define the {bin weights} at time $t$ as 
\begin{equation*}
\omega_t(u) = \sum_{i:\xi_t^i \in u} \omega_t^i, \qquad u \in {\mathcal B},
\end{equation*}
where the empty sum is zero (so an unoccupied bin $u$ has $\omega_t(u) = 0$).

For the parent $\xi_t^i$ of $\hat{\xi}_t^j$, we write $\text{par}(\hat{\xi}_t^j) = \xi_t^i$. 
A child is just a copy of its parent:
$$\text{par}(\hat{\xi}_t^j) = \xi_t^i \text{ implies }\hat{\xi}_t^j = \xi_t^i.$$ 
{\em Each child has a {unique} parent}.
Setting the number of children 
of each parent
completely defines 
the children, as the 
choices of the children's 
indices (the $j$'s in $\hat{\xi}_t^j$) do not matter.
The number of children of 
$\xi_t^i$ will be written $C_t^i$: 
\begin{equation}\label{children}
C_t^i = \#\{j: \textup{par}(\hat{\xi}_t^j) = \xi_t^i\}, 
\end{equation}
where $\#S =$ number of elements in a set $S$.

Recall that $N_t(u)$ is the number of 
children in bin $u$ at time $t$. We require that 
there is at least one child in each occupied bin, no children in unoccupied bins, and $N$ total 
children at each time $t$. Thus,
\begin{equation}\label{Ntu_cond}
N_t(u) \ge 1 \text{ if } \omega_t(u)>0, \qquad N_t(u)= 0 \text{ if } \omega_t(u) = 0, \qquad \sum_{u \in {\mathcal B}}N_t(u) = N.
\end{equation}

We write ${\mathcal F}_t$ for the $\sigma$-algebra generated by the  weighted ensemble up to, but not including, the $t$-th selection step. We will assume the particle allocation is known before selection. 
Similarly, we write $\hat{\mathcal F}_t$ 
for the $\sigma$-algebra generated by 
weighted ensemble up to and including 
the $t$-th selection step. In detail, 
\begin{align*}
{\mathcal F}_t &= \sigma\left((\xi_s^i, \omega_s^i)_{0 \le s \le  t}^{i=1,\ldots,N},N_s(u)_{0 \le s \le t}^{u \in {\mathcal B}},(\hat{\xi}_s^i,\hat{\omega}_s^i)_{0 \le s \le t-1}^{i=1,\ldots,N},(C_s^i)_{0 \le s \le t-1}^{i=1,\ldots,N}\right)\\
\hat{\mathcal F}_t &= \sigma\left((\xi_s^i, \omega_s^i)_{0 \le s \le  t}^{i=1,\ldots,N},N_s(u)_{0 \le s \le t}^{u \in {\mathcal B}},(\hat{\xi}_s^i,\hat{\omega}_s^i)_{0 \le s \le t}^{i=1,\ldots,N},(C_s^i)_{0 \le s \le t}^{i=1,\ldots,N}\right).
\end{align*}

\begin{algorithm}\caption{Weighted ensemble 
}

Pick initial parents and weights $(\xi_0^i,\omega_0^i)^{i=1,\ldots,N_{\textup{init}}}$ with $\sum_{i=1}^{N_{\textup{init}}} \omega_0^i = 1$, choose a collection ${\mathcal B}$ of bins, and define a final time $T$. Then for $t \ge 0$, iterate the following:
\vskip5pt

\begin{itemize}[leftmargin=20pt]
\item {\em (Selection step)}
Each parent $\xi_t^i$ is assigned a number 
$C_t^i$ of children, as follows:

\vskip5pt

In each occupied bin $u \in {\mathcal B}$, conditional on ${\mathcal F}_t$, let $(C_t^i)^{i:\textup{bin}(\xi_t^i) = u}$ be 
$N_t(u)$ samples from the distribution
$\{\omega_t^i/\omega_t(u):\,\xi_t^i \in u\}$, where $N_t(u)^{u \in {\mathcal B}}$ satisfies~\eqref{Ntu_cond}.
The children $(\hat{\xi}_t^i)^{i=1,\ldots,N}$ are  defined by~\eqref{children}, with weights
\begin{equation}\label{omegati}
{\hat \omega}_t^i = \frac{\omega_t(u)}{N_t(u)}, \qquad \textup{if }{\hat \xi}_t^i \in u.
\end{equation}
Selections in distinct bins are conditionally independent.
\vskip2pt
\item {\em (Mutation step)} 
Each child $\hat{\xi}_t^i$ independently  evolves one time step. Thus:
\vskip5pt
Conditionally on $\hat{\mathcal F}_t$, the children $(\hat{\xi}_t^i)^{i=1,\ldots,N}$ evolve independently according to the Markov kernel $K$, becoming the next parents $({\xi}_{t+1}^i)^{i=1,\ldots,N}$, 
with weights
\begin{equation}\label{weight_mut}
\omega_{t+1}^i = {\hat \omega}_t^i, \qquad i=1,\ldots,N.
\end{equation}
Then time advances,
$t \leftarrow t+1$. Stop if $t = T$, else 
return to the selection~step.
\end{itemize}
\vskip5pt
Algorithm~\ref{alg_opt_allocation} outlines an optimization for the allocation $N_t(u)^{u \in {\mathcal B}}$, and
a procedure for choosing the bins ${\mathcal B}$ is in Algorithm~\ref{alg_opt_bins}.  
For the initialization, see Algorithm~\ref{alg_initialize}.
\label{alg1}
\end{algorithm}

In Algorithm~\ref{alg1},
we do not explicitly say how 
we sample the $N_t(u)$ children in 
each bin $u$. 
Our framework below allows 
for any unbiased resampling 
scheme. We give a selection variance 
formula that assumes residual resampling; 
see Lemma~\ref{lem_sel_var} and 
Algorithm~\ref{alg_residual}. 
Residual resampling 
has performance on par with other 
standard resampling methods like 
systematic 
and stratified resampling~\cite{douc2005comparison}. See~\cite{webber2} 
for more details on resampling in 
the context of sequential Monte Carlo.

\subsection{Ergodic averages}

We are interested in using Algorithm~\ref{alg1} to estimate $$\theta_T \approx \int f\,d\mu$$ where
we recall $\mu$ is the stationary distribution of the Markov kernel $K$, and
\begin{equation}\label{theta_T}
\theta_T = \frac{1}{T}\sum_{t=0}^{T-1} \sum_{i=1}^N \omega_t^i f(\xi_t^i).
\end{equation}

Note that $\theta_T$ 
is simply the running average of $f$ over the 
weighted ensemble up to time $T-1$.
In particular,~\eqref{theta_T} is {\em not} 
a time average over ancestral lines that 
survive up to time $T-1$, but rather it is 
an average over the weighted 
ensemble at each time $0 \le t \le T-1$. 
Our time 
averages~\eqref{theta_T} require 
no replica storage and should have smaller 
variances than averages over surviving 
ancestral lines~\cite{aristoff2019ergodic}.

\subsection{Consistency results}
The next results, from a 
companion article~\cite{aristoff2019ergodic}, 
show that weighted ensemble is {unbiased} and that 
weighted ensemble time 
averages converge. 
The latter does 
not in general follow 
from the former, 
as standard unbiased particle methods 
such as sequential Monte Carlo can 
have variance explosion~\cite{aristoff2019ergodic}. 
(The proofs in~\cite{aristoff2019ergodic} have 
$N_{\textup{init}} = N$, but they are 
easily modified for $N_{\textup{init}} \ne N$.)

\begin{theorem}[From~\cite{aristoff2019ergodic}]\label{thm_unbiased}
In Algorithm~\ref{alg1}, 
suppose that the initial particles 
and weights are distributed 
as $\nu$, in the sense that
\begin{equation}\label{initialization}
{\mathbb E}\left[\sum_{i=1}^{N_{\textup{init}} }\omega_0^i g(\xi_0^i)\right] = \int g\,d\nu
\end{equation}
for all real-valued bounded functions $g$ on state space. Let $(\xi_t)_{t \ge 0}$ be a Markov chain 
with kernel $K$ and initial distribution 
$\xi_0 \sim \nu$. Then 
for any time $T > 0$,
\begin{equation*}
{\mathbb E}\left[\sum_{i=1}^{N} \omega_T^i g(\xi_T^i)\right] = {\mathbb E}[g(\xi_T)]
\end{equation*}
for all real-valued bounded functions $g$ on state space.
\end{theorem}

Recall we assume $K$ is uniformly geometrically ergodic~\cite{doucbook}.
\begin{theorem}[From~\cite{aristoff2019ergodic}] 
\label{thm_ergodic}
Weighted ensemble is ergodic 
in the following sense:
 $$\lim_{T \to \infty} \theta_T \stackrel{a.s.}{=} \int f\,d\mu.$$
\end{theorem}
Theorem~\ref{thm_unbiased} does not use ergodicity of $K$, though Theorem~\ref{thm_ergodic} obviously does.

\subsection{Variance analysis}\label{sec:variance}

We will make use of the following 
analysis from~\cite{aristoff2019ergodic} concerning 
the weighted ensemble variance.
Define the Doob martingales~\cite{doob1940regularity}
\begin{equation}\label{doob_mart}
D_t = {\mathbb E}[\theta_T|{\mathcal F}_t], \qquad \hat{D}_t = {\mathbb E}[\theta_T|\hat{\mathcal F}_t].
\end{equation}

The Doob decomposition in Theorem~\ref{thm_doob} below filters the 
variance through the $\sigma$-algebras ${\mathcal F}_t$ and $\hat{\mathcal F}_t$. 
It is a way to decompose the variance 
into contributions from the initial condition 
and each time step. This type of telescopic variance decomposition is standard 
in sequential Monte Carlo, 
although it is usually applied at 
the level of measures on state space, 
which corresponds to the 
infinite particle limit $N \to \infty$~\cite{del2004feynman}. 
We use the finite $N$ formula 
directly to minimize variance, building on ideas in~\cite{aristoff2016analysis}. 
\begin{theorem}[From~\cite{aristoff2019ergodic}]\label{thm_doob}
By Doob decomposition,
\begin{align}
\theta_T^2 - {\mathbb E}[\theta_T]^2 + R_T &= 
\left(D_0 - {\mathbb E}[\theta_T]\right)^2 
\label{var0} \\
&\quad+ \sum_{t=1}^{T-1}{\mathbb E}\left[\left.\left(D_t-\hat{D}_{t-1}\right)^2\right|\hat{\mathcal F}_{t-1}\right] \label{var1}\\
&\qquad + 
\sum_{t=1}^{T-1}{\mathbb E}\left[\left.\left(\hat{D}_{t-1}-{D}_{t-1}\right)^2\right|{\mathcal F}_{t-1}\right], \label{var2}
\end{align}
where $R_T$ is mean-zero, ${\mathbb E}[R_T] = 0$.
\end{theorem}

The terms on the right-hand side of~\eqref{var0},\eqref{var1}, and~\eqref{var2} yield
the contributions to the variance of $\theta_T$ 
from, respectively, the initial condition, mutation steps, 
and selection steps of Algorithm~\ref{alg1}. 
We thus refer to the summands of~\eqref{var1} and~\eqref{var2} as the 
{\em mutation variance} and {\em selection variance} 
of weighted ensemble.

Below, define 
\begin{equation}\label{ht}
h_{t} = \sum_{s=0}^{T-t-1} K^s f,
\end{equation}
and for any function $g$ 
and probability distribution $\eta$ on 
state space, let
\begin{equation}\label{var}
\textup{Var}_\eta g := \int g^2(\xi)\eta(d\xi) - \left(\int g(\xi)\eta(d\xi)\right)^2.
\end{equation}
Above and below, the dependence of $D_t$, ${\hat D}_t$ and $h_t$ on $T$ is suppressed.
\begin{lemma}[From~\cite{aristoff2019ergodic}]\label{lem_mut_var}
The mutation variance 
at time $t$ is 
\begin{align}\label{eq_mutvar}
{\mathbb E}\left[\left.\left(D_{t+1}-\hat{D}_{t}\right)^2\right|\hat{\mathcal F}_{t}\right] = \frac{1}{T^2}\sum_{i=1}^N \left(\hat{\omega}_t^i\right)^2\textup{Var}_{K(\hat{\xi}_t^i,\cdot)}h_{t+1}.
\end{align}
\end{lemma}
To formulate the selection variance, we define
\begin{equation*}
\delta_t^i = \frac{N_t(u)\omega_t^i}{\omega_t(u)} - \left\lfloor\frac{N_t(u)\omega_t^i}{\omega_t(u)} \right\rfloor \qquad \text{if }\xi_t^i \in u, \qquad \delta_t(u) = \sum_{i: \xi_t^i \in u} \delta_t^i,
\end{equation*}
where $\lfloor x\rfloor$ is the floor function, or the greatest integer less than or equal to $x$. In Lemma~\ref{lem_sel_var}, we assume that the $(C_t^j)^{j=1,\ldots,N}$ in 
Algorithm~\ref{alg1}
are obtained using
residual resampling. See~\cite{douc2005comparison,webber2} 
or Algorithm~\ref{alg_residual} in the Appendix below for a description of residual resampling.
\begin{lemma}[From~\cite{aristoff2019ergodic}]\label{lem_sel_var}
The 
selection variance at time $t$ is 
\begin{align}\begin{split}\label{sel_varnew}
&{\mathbb E}\left[\left.\left(\hat{D}_{t}-{D}_{t}\right)^2\right|{\mathcal F}_{t}\right] = \frac{1}{T^2}
\sum_{u\in {\mathcal B}} \left(\frac{\omega_t(u)}{N_t(u)}\right)^2\delta_t(u)\textup{Var}_{\eta_t(u)}Kh_{t+1},\\
& \qquad \eta_t(u) := \sum_{i:\textup{bin}(\xi_t^i) = u} \frac{\delta_t^i}{\delta_t(u)}\delta_{\xi_t^i}
\end{split}
\end{align}
\end{lemma}
By definition~\eqref{var}, the variances in Lemmas~\ref{lem_mut_var} and~\ref{lem_sel_var} 
rewrite as
\begin{align*}
&\textup{Var}_{K(\hat{\xi}_t^i,\cdot)}h_{t+1} = Kh_{t+1}^2(\hat{\xi}_t^i) - (Kh_{t+1}(\hat{\xi}_t^i))^2,\\ 
&\textup{Var}_{\eta_t(u)}Kh_{t+1} = \sum_{i:\textup{bin}(\xi_t^i) = u} \frac{\delta_t^i}{\delta_t(u)}(Kh_{t+1}(\xi_t^i))^2- \left(
\sum_{i:\textup{bin}(\xi_t^i) = u} \frac{\delta_t^i}{\delta_t(u)}Kh_{t+1}(\xi_t^i)\right)^2.
\end{align*}

\subsection{Poisson equation}

Because we are interested in large 
time horizons $T$, below we will consider the mutation and selection variances in the limit 
$T \to \infty$. We will see that for any probability distribution $\eta$ on state space,
\begin{align*}
\lim_{T \to \infty} \textup{Var}_\eta h_{t+1} &= \textup{Var}_\eta h\\ \lim_{T \to \infty} \textup{Var}_\eta Kh_{t+1} &= \textup{Var}_\eta Kh,
\end{align*}
where $h$ is the solution to the Poisson 
equation~\cite{lelievre2016partial,nummelin}
\begin{equation}\label{h}
({Id}-K)h = f - \int f \,d\mu, \qquad \int h\,d\mu = 0,
\end{equation}
where $Id = $ the identity kernel. 
Existence and uniqueness of the solution $h$ easily follow from uniform geometric 
ergodicity of the Markov kernel $K$. 
Indeed, if
$(\xi_t)_{t \ge 0}$ is a Markov chain 
with kernel $K$, then we can write
\begin{align}\begin{split}\label{h2}
h(\xi) &=  \sum_{t=0}^\infty \left(K^tf(\xi) - \int f\,d\mu\right)\\
&= \sum_{t=0}^\infty \left({\mathbb E}[f(\xi_t)|\xi_0 = \xi]- {\mathbb E}[f(\xi_t)|\xi_0 \sim \mu]\right),
\end{split}
\end{align}
where $\xi_0 \sim \mu$ indicates $\xi_0$ is initially distributed according 
to the steady state $\mu$ of $K$.
Uniform geometric ergodicity and the Weierstrass $M$-test show 
that the sums in~\eqref{h2} converge 
absolutely and
uniformly. As a consequence, $h$ in~\eqref{h2} 
solves~\eqref{h}.

Interpreting~\eqref{h2}, $h(\xi)$ is the 
mean discrepancy 
of a time average of $f(\xi_t)$ 
starting at $\xi_0 = \xi$
with a time average of $f(\xi_t)$ 
starting at steady state $\xi_0 \sim \mu$. 
This discrepancy in 
the time averages
has  
been normalized so that it 
has a nontrivial limit, and in particular
does not vanish, as time goes 
to infinity.

The Poisson solution $h$ is critical for understanding 
and estimating the weighted ensemble 
variance. Besides, $h$ can be used to 
identify model features, such as sets that are {\em metastable}
for the underlying Markov chain 
defined by $K$, as well as narrow 
pathways 
between these sets. Metastable 
sets are, roughly speaking, regions of 
state space in which 
the Markov chain tends 
to become trapped.

To understand the behavior 
of $h$, we define {metastable 
sets} more precisely. A region $R$ 
in state space is 
metastable for $K$ 
if $(\xi_t)_{t \ge 0}$ tends to 
equilibrate in $R$ much faster 
than it escapes from $R$. 
The rate of equilibration in 
$R$ can be understood 
in terms of the {\em quasistationary 
distribution}~\cite{QSD} in $R$. 
See {\em e.g.}~\cite{lelievre2012two} 
for more discussion on metastablity.
The Poisson solution $h$ 
tends to be nearly constant 
in regions that are metastable 
for $K$. This is 
because the mean discrepancy 
in a time average of $f(\xi_t)$ over 
two copies of
$(\xi_t)_{t \ge 0}$ with different
starting points in the 
same metastable set $R$ is small: 
both copies 
tend to reach the same
 quasistationary distribution
in $R$ 
before escaping from $R$. In 
the regions between 
metastable sets, however, $h$ 
tends to have large 
variance. If $f$ is a 
characteristic or indicator 
function, this 
variance tends to be 
larger the closer these 
regions are to the support 
of $f$ (the set where $f = 1$). More 
generally, the variance of $h$ 
is larger near regions $R$ 
where the stationary 
average $\int_R f\,d\mu$ of 
$f$ is large. See 
Figure~\ref{fig_V_h_pi} 
for illustration of 
these features.

\section{Minimizing the variance}\label{sec:var_min}

Our strategy for minimizing 
the variance is based on {\em 
choosing the particle allocation to 
minimize mutation variance} 
and {\em picking the bins to 
mitigate selection variance}. 
Minimizing mutation variance 
is closely connected with minimizing 
a certain sensitivity; see 
the Appendix below. Both 
strategies require some 
coarse estimates of $K$ 
and $h$. We propose using 
ideas from Markov state modeling 
to construct {\em microbins} 
from which we estimate $K,h$. 
The microbins can be significantly 
smaller 
than the weighted ensemble bins, 
as we discuss below.

\subsection{Resampling times}\label{sec:resample}

Recall that weighted ensemble is fully characterized 
by the choice of resampling times, bins, and particle 
allocation. Though we focus on the latter two here, 
we briefly comment on the former. The 
resampling times are implicit in our framework. We 
assume here that  
$K = K_{\Delta t}$ is a $\Delta t$-skeleton of an underlying Markov process, or a sequence of values of 
the underlying process at time intervals $\Delta t$. 
In this setup, $\Delta t$ is a fixed 
resampling time, and we are ignoring 
the time discretization. 
(Actually, the resampling times 
need not be fixed -- they can 
be any times at which the underyling 
process has the strong Markov property~\cite{aristoff2016analysis}. 
In practice, the underlying 
process must be discretized 
in time, and weighted 
ensemble is used with 
the discretized process.)

Suppose that the underlying Markov process is one 
of the ones mentioned in 
the Introduction: either Langevin dynamics, or
reaction network modeled by a continuous time Markov chain 
on finite 
state space. 
Suppose moreover that 
microbins in continuous state space are 
domains with piecewise smooth 
boundaries (for instance, Voronoi 
regions), and that the bins are unions of microbins. Then the 
underlying process does not cross 
between distinct microbins, or between distinct bins, infinitely often in finite time. As a result, weighted ensemble 
should not degenerate 
in the 
limit as $\Delta t \to 0$, 
as we now show.

Consider the variance from selection 
in Lemma~\ref{lem_sel_var}. By~\eqref{omegati}, the 
weights of all the children in each 
bin $u \in {\mathcal B}$ 
are all equal to $\omega_t(u)/N_t(u)$ after 
the selection step. If $\Delta t$ is 
very small, then almost none 
of the children move to different 
microbins 
or bins in the mutation step. 
If exactly zero of the children change bins, then for residual resampling in the 
selection step at the next time $t$, provided the allocation $N_t(u)^{u \in {\mathcal B}}$ has 
not changed, we have $\delta_t^i = 0$ for all $i=1,\ldots,N$. (Note that 
with our optimal allocation strategy 
in Algorithm~\ref{alg_opt_allocation}, 
if we avoid unnecessary  
resampling of $R_t(u)^{u \in {\mathcal B}}$, 
then
$N_t(u)^{u \in {\mathcal B}}$ 
does not change unless particles 
move between microbins.) Thus from Lemma~\ref{lem_sel_var} there 
is {zero selection variance}, and 
in fact no resampling occurs in 
the selection step. 
Provided particles do not cross 
bins or microbins infinitely often in finite time, 
and the allocation only changes when 
particles move between microbins, 
this suggests
there is {\em no variance blowup 
when $\Delta t \to 0$}. 
We expect then that the frequency $\Delta t$ of resampling 
should be driven not by variance 
cost but by computational cost, 
{\em e.g.} processor communication cost.

\subsection{Minimizing mutation variance}

The mutation variance depends on 
the choice of weighted ensemble bins as well as
the particle allocation at each time $t$. 
In this section we focus on the particle 
allocation for an arbitrary choice ${\mathcal B}$ 
of bins. To understand this relationship between the allocation and mutation variance, following 
ideas from~\cite{aristoff2016analysis}, {\em we look at the mutation 
variance visible {before} selection}. 
It is so named because, unlike the 
mutation variance in Lemma~\ref{lem_mut_var},
it is a function of quantities $\omega_t(u)$, $N_t(u)$, $(\omega_t^i,\xi_t^i)^{i=1,\ldots,N}$ 
that are known at time $t$ before selection.
\begin{proposition}\label{prop_vis_var}
The mutation variance visible before selection satisfies 
\begin{equation*}
\lim_{T \to \infty} T^2{\mathbb E}\left[\left.\left(D_{t+1}-\hat{D}_{t}\right)^2\right|{\mathcal F}_{t}\right]  = \sum_{u \in {\mathcal B}}\frac{\omega_t(u)}{N_t(u)} \sum_{i: {\xi}_t^i \in u}\omega_t^i \textup{Var}_{K(\xi_t^i,\cdot)} h, 
\end{equation*}
where $h$ is defined in~\eqref{h}.
\end{proposition}
\begin{proof}
By definition of the selection step 
(see Algorithm~\ref{alg1}), 
\begin{equation}\label{betati}
{\mathbb E}[C_t^i|{\mathcal F}_t] = \frac{N_t(u)\omega_t^i}{\omega_t(u)}.
\end{equation}
 From Lemma~\ref{lem_mut_var}, 
\begin{align}\begin{split}\label{vis_mut_var}
&T^2{\mathbb E}\left[\left.\left(D_{t+1}-\hat{D}_{t}\right)^2\right|{\mathcal F}_{t}\right] 
=  {\mathbb E}\left[\left.\sum_{i=1}^N (\hat{\omega}_t^i)^2 \textup{Var}_{K(\hat{\xi}_t^i,\cdot)}h_{t+1}\right|{\mathcal F}_t\right] \\
&=  \sum_{u \in {\mathcal B}}\left(\frac{\omega_t(u)}{N_t(u)}\right)^2 {\mathbb E}\left[\left.\sum_{i: \hat{\xi}_t^i \in u} \textup{Var}_{K(\hat{\xi}_t^i,\cdot)}h_{t+1}\right|{\mathcal F}_t\right] \qquad \text{(using }\eqref{omegati}\text{)}\\
&=  \sum_{u \in {\mathcal B}}\left(\frac{\omega_t(u)}{N_t(u)}\right)^2 \sum_{i: {\xi}_t^i \in u}{\mathbb E}\left[\left.\sum_{j: \textup{par}(\hat{\xi}_t^j)  = \xi_t^i} [Kh_{t+1}^2(\hat{\xi}_t^j) - (Kh_{t+1}(\hat{\xi}_t^j))^2]\right|{\mathcal F}_t\right] \\
&=  \sum_{u \in {\mathcal B}}\frac{\omega_t(u)}{N_t(u)} \sum_{i: {\xi}_t^i \in u}\omega_t^i[Kh_{t+1}^2({\xi}_t^i) - (Kh_{t+1}({\xi}_t^i))^2] \qquad \text{(using }\eqref{betati}\text{)}.
\end{split}
\end{align}
In light of~\eqref{ht} and~\eqref{h2}, and using the fact that $$\textup{Var}_{K(\xi_t^i,\cdot)}h = Kh^2({\xi}_t^i) - (Kh({\xi}_t^i))^2,$$ 
we get the result from letting $T \to \infty$.
\end{proof}

We let $T \to \infty$ since
we are interested in long-time averages. 
The simpler formulas that result, 
as they involve $h$ instead of $h_t$, allow 
for strategies to estimate fixed
optimal bins {before beginning weighted 
ensemble simulations}, which we discuss 
more below. Thus for minimizing the limiting mutation variance, we consider the following optimization: 
\begin{align}\begin{split}\label{opt_allocation}
&\text{minimize }  \sum_{u \in {\mathcal B}}\frac{\omega_t(u)}{N_t(u)} \sum_{i: {\xi}_t^i \in u}\omega_t^i\textup{Var}_{K(\xi_t^i,\cdot)}h,\\
&\qquad  \text{ over all choices of }N_t(u) \in {\mathbb R}^+ \text{ such that }\sum_{u \in {\mathcal B}} N_t(u) = N,
\end{split}
\end{align}
where the bins ${\mathcal B}$ are fixed, 
and we temporarily allow the 
allocation $N_t(u)^{u \in {\mathcal B}}$ to be noninteger. A  Lagrange 
multiplier calculation shows that the solution to~\eqref{opt_allocation} is
\begin{equation}\label{Ntu}
N_t(u) = \frac{N\sqrt{\omega_t(u)\sum_{i:\xi_t^i \in u} \omega_t^i \textup{Var}_{K(\xi_t^i,\cdot)}h}}{\sum_{u \in {\mathcal B}}\sqrt{\omega_t(u)\sum_{i:\xi_t^i \in u} \omega_t^i \textup{Var}_{K(\xi_t^i,\cdot)}h}},
\end{equation}
provided the denominator above is nonzero.
{\em Note this solution is idealized, as $N_t(u)$ 
must always be an integer, and $h$ 
and $K$ are not known exactly.} We explain a practical 
implementation of~\eqref{Ntu} in
Algorithm~\ref{alg_opt_allocation}.

Our choice of the particle allocation will 
be based on~\eqref{Ntu}; see Algorithm~\ref{alg_opt_allocation}. Notice that at each time $t$ we only minimize one 
term, the summand in~\eqref{var1} corresponding to the mutation variance at time $t$, in the Doob decomposition in Theorem~\ref{thm_doob}. Later, when we optimize 
bins, we will optimize the summand in~\eqref{var2} corresponding to the 
selection variance at time $t$. 
In particular, we only minimize the
mutation and selection variances 
at the current time, and not the sum 
of these variances over all times.
In the $T \to \infty$ limit, 
we expect that
weighted ensemble reaches a steady state, 
provided the bin choice and allocation 
strategy ({\em e.g.} from Algorithms~\ref{alg_opt_allocation} and~\ref{alg_opt_bins}) 
do not change over time.
If the weighted ensemble indeed 
reaches a steady state, then the mutation 
and selection variances become stationary in $t$, making it reasonable to minimize them 
only at the current time. (Under 
appropriate conditions,
the variances in~\eqref{var1}-\eqref{var2} should 
also become {independent} over time $t$ in the 
$N \to \infty$ limit.  
See~\cite{del2004feynman} for 
related results 
in the context of sequential Monte Carlo.)

Note the term $\textup{Var}_{K(\xi_t^i,\cdot)}h = Kh^2(\xi_t^i) - (Kh(\xi_t^i))^2$ appearing in~\eqref{Ntu}. As 
discussed in Section~\ref{sec:notation}, {\em this 
variance tends to be large 
in regions between metastable 
sets}. The optimal 
allocation~\eqref{Ntu} 
favors putting children 
in such regions, increasing 
the likelihood that their 
descendants will visit 
both the adjacent metastable 
sets. See the Appendix for a 
connection between minimizing mutation 
variance and minimizing the 
sensitivity of the stationary 
distribution $\mu$ to perturbations 
of $K$.

\subsection{Mitigating selection variance}\label{sec:mit_sel}

We begin by observing that 
if bins are small, then 
so is the selection variance.
In particular, if each 
bin is a single point in state space,
then Lemma~\ref{lem_sel_var} 
shows that the selection 
variance is {\em zero}. One way, 
then, to get small selection 
variance is to have a lot of bins.
When simulations of the 
underlying Markov chain
are very expensive, however, 
we cannot afford a large 
number of bins; see Remark~\ref{rmk_bins}  
in Section~\ref{sec:params} below. As a result the 
bins are not so 
small, and we investigate 
the selection variance 
to 
decide how to construct them.

\begin{lemma}\label{lem_sel_var2}
The selection variance 
at time $t\ge 1$ satisfies
\begin{align}\begin{split}\label{sel_var2}
&\lim_{T \to \infty}T^2{\mathbb E}\left[\left.\left(\hat{D}_{t}-{D}_{t}\right)^2\right|{\mathcal F}_{t}\right] = \sum_{u \in {\mathcal B}} \left(\frac{\omega_t(u)}{N_t(u)}\right)^2\delta_t(u)\textup{Var}_{\eta_t(u)}Kh,
\end{split}
\end{align}
where $\eta_t(u)$ is defined in Lemma~\ref{lem_sel_var} and $h$ is defined in~\eqref{h}.
\end{lemma}
\begin{proof}
From Lemma~\ref{lem_sel_var},~\eqref{ht} and~\eqref{h2},  letting $T \to \infty$ gives the result.
\end{proof}

We choose to mitigate selection variance by choosing 
bins so that~\eqref{sel_var2} is small. This likely
requires some sort of search in bin space.
For simplicity, we assume that the bins 
do not change in time and are chosen at 
the start of Algorithm~\ref{alg1}. 
Of course it is possible
to update bins adaptively, 
at a frequency that 
depends on how the cost of bin 
searches compares to that of 
particle evolution.

We will minimize an {\em agnostic} variant 
of~\eqref{sel_var2}, 
for which we make no assumptions about 
$N_t(u)$, $\omega_t(u)$ and $\delta_t(u)$. 
This allows us to optimize 
fixed bins using a 
time-independent 
objective function, without 
taking the particle allocation into 
account.
Our agnostic optimization also 
is not specific to residual 
resampling. Indeed, though the precise formula 
for the selection variance depends on 
the resampling method, 
the selection variance should 
always contain terms of the 
form $\text{Var}_\eta Kh$, 
where $\eta$ are probability 
distributions in the individual 
weighted ensemble bins; see~\cite{aristoff2019ergodic}. It 
is exactly these terms that
we choose to minimize.

Thus for our agnostic selection variance minimization, we let
$$\eta_u^{\textup{unif}}(d\xi) = \frac{\mathbbm{1}_{\xi \in u}\,d\xi}{\int \mathbbm{1}_{\xi \in u}\,d\xi}$$ 
be the 
uniform distribution in 
bin $u \in {\mathcal B}$, and consider the following problem:
\begin{align}\begin{split}\label{opt_bins}
&\text{minimize } \sum_{u \in {\mathcal B}}\text{Var}_{\eta_u^{\textup{unif}}} Kh  \\
& \qquad \text{over choices of } {\mathcal B,} \text{ subject to the constraint } \#{\mathcal B} = M.
\end{split}
\end{align}
Like~\eqref{opt_allocation}, {\em this is 
idealized because we cannot 
directly access $K$ or $h$.} We 
describe a practical implementation of~\eqref{opt_bins} 
in Algorithm~\ref{alg_opt_bins}. 
Informally, solutions to~\eqref{opt_bins} are characterized 
by the property that, inside each 
individual bin $u$, 
the value of $Kh$ does 
not change very much.

Our choice of bins is based on~\eqref{opt_bins}. Here $M$ is the desired total number of bins. 
Our agnostic perspective leads us to 
use the uniform distribution $\eta_u^{\textup{unif}}$ in 
each bin $u$. When bins are formed 
from combinations of a fixed collection 
of microbins,~\eqref{opt_bins} is a discrete optimization problem that 
usually lacks a closed form solution. 
Algorithm~\ref{alg_opt_bins} below 
solves a discrete version of~\eqref{opt_bins} by simulated 
annealing. 
Because of the similarity of~\eqref{opt_bins} with the 
$k$-means 
problem~\cite{kmeans}, we expect that 
there are
more efficient methods, 
but this is not the focus 
of the present work.

\section{Microbins and parameter choice}\label{sec:params}

To approximate the solutions to~\eqref{opt_allocation} 
and~\eqref{opt_bins}, we use {\em microbins} to gain information about $K$ and $h$. 
The collection ${\mathcal M}{\mathcal B}$ of 
microbins is a
finite partition of state space that refines 
the weighted ensemble bins  
${\mathcal B}$, in the sense that 
every element of ${\mathcal B}$ 
is a union of elements of ${\mathcal M}{\mathcal B}$.
Thus each bin is comprised of a number of microbins, 
and each microbin is inside exactly one bin.
The idea is to use exploratory simulations, 
over short time horizons, to approximate $K$ and $h$ by observing transitions 
between microbins.

In more detail,  
we estimate the probability to transition from microbin 
$p \in {\mathcal M}{\mathcal B}$ to microbin $q \in {\mathcal M}{\mathcal B}$ by a matrix $\tilde{K} = (\tilde{K}_{pq})$,
\begin{equation}\label{tildeK}
\tilde{K}_{pq} \approx \dfrac{ \iint \nu(d\xi)K(\xi,d\xi')\mathbbm{1}_{\xi \in p,\,\xi' \in q}}{\int \nu(d\xi)\mathbbm{1}_{\xi \in p}},
\end{equation}
and we estimate $f$ on microbin $p \in {\mathcal M}{\mathcal B}$ by a vector $\tilde{f} = (\tilde{f}_p)$,
\begin{equation}\label{tildef}
\tilde{f}_p \approx \frac{\int f(\xi)\mathbbm{1}_{\xi \in p}\,\nu(d\xi)}{\int \mathbbm{1}_{\xi \in p}\nu(d\xi)}.
\end{equation}
Here $\nu$ is some convenient measure, 
for instance an empirical measure 
obtained from preliminary weighted 
ensemble simulations. 
This strategy echoes work in the Markov state model community~\cite{husic2018markov,pande2010everything,
sarich2010approximation,schutte2013metastability}.

For small enough microbins, 
we could replace $\nu$ in~\eqref{tildeK} with any other measure without changing too much 
the value of the estimates on the  
right hand sides of~\eqref{tildeK} and~\eqref{tildef}. 
Moreover, if $f$ is the characteristic function of a 
microbin, then
$\nu$ could be replaced with 
any measure supported in microbin $p$ {\em without changing at all} the value of
the right hand side of~\eqref{tildef}. 
This is the case for the mean first passage time 
problem mentioned in the 
Introduction, and fleshed out in 
the numerical example in Section~\ref{sec:numerics}: there, $f$ is 
the characteristic function 
of the target set, which can 
be chosen to be a microbin.

\begin{algorithm}\caption{Optimizing the particle allocation}
Given the particles and weights $(\xi_0^i,\omega_0^i)^{i=1,\ldots,N_{\textup{init}}}$ at $t = 0$, or $(\xi_t^i,\omega_t^i)^{i=1,\ldots,N}$ at $t\ge 1$:
\begin{itemize}[leftmargin=20pt]
\item
Define the 
following approximate solution to~\eqref{opt_allocation}:
\begin{equation}\label{star}
\tilde{N}_t(u) = \frac{N\sqrt{\omega_t(u) \sum_{i:\xi_t^i \in u} \omega_t^i [{\tilde K}{\tilde h}^2 - ({\tilde K}\tilde{h})^2]_{p(\xi_t^i)}}}{\sum_{u \in {\mathcal B}}\sqrt{\omega_t(u) \sum_{i:\xi_t^i \in u} \omega_t^i [{\tilde K}{\tilde h}^2 - ({\tilde K}\tilde{h})^2]_{p(\xi_t^i)}}},
\end{equation}
where $p(\xi_t^i) \in {\mathcal M}{\mathcal B}$ is the microbin containing $\xi_t^i$.

\item 
Let $\tilde{N}$ count the occupied bins,
\begin{equation*}
\tilde{N} = \sum_{u \in {\mathcal B}} \mathbbm{1}_{\omega_t(u)>0}.
\end{equation*}
\item Let $R_t(u)^{u \in {\mathcal B}}$ be $N-{\tilde N}$ samples from the distribution 
$\{\tilde{N}_t(u)/N: u \in {\mathcal B}\}$. \item In Algorithm~\ref{alg1}, define the particle allocation as
\begin{equation*}
N_t(u) = {\mathbbm{1}}_{\omega_t(u)>0} + R_t(u).
\end{equation*}
\end{itemize}
If the denominator of~\eqref{star} is $0$,  we set $N_t(u) =  \#\{i:\xi_t^i \in u\}$. 
\label{alg_opt_allocation}
\end{algorithm}

With $\tilde{f}$ and the microbin-to-microbin transition matrix $\tilde{K}$ in hand, we can 
obtain an approximate solution $\tilde{h}$ to the Poisson equation~\eqref{h}, 
simply by replacing $K$, $f$ and $\mu$ in that equation 
with, respectively, $\tilde{K}$, $\tilde{f}$ and the stationary 
distribution $\tilde{\mu}$ of $\tilde{K}$ (we assume $\tilde{K}$ is aperiodic and irreducible). 
That is, $\tilde{h}$ solves
\begin{equation}\label{tildeh}
(\tilde{I} - \tilde{K})\tilde{h} = \tilde{f} - \tilde{f}^T\tilde{\mu}\tilde{\mathbbm{1}}, \qquad {\tilde h}^T\tilde{\mu} = 0,
\end{equation}
where $\tilde{I}$ and $\tilde{\mathbbm{1}}$ 
are the identity matrix and all ones 
column vector of the appropriate sizes, 
and $\tilde{f}$, $\tilde{\mu}$ and $\tilde{h}$ 
are column vectors.
Then we can approximate the solutions to~\eqref{opt_allocation} and~\eqref{opt_bins}
by simply replacing $K$ and $h$ in those optimization 
problems 
with $\tilde{K}$ and $\tilde{h}$. 
See Algorithms~\ref{alg_opt_allocation} and~\ref{alg_opt_bins} 
for details. 
We can also use 
$\tilde{\mu}$ to 
initialize weighted 
ensemble; see 
Algorithm~\ref{alg_initialize}.

We have in mind that the microbins are constructed using 
ideas from Markov state modeling~\cite{husic2018markov,pande2010everything,
sarich2010approximation,schutte2013metastability}. In this setup, 
the microbins are simply the Markov states. 
These could be determined from a clustering analysis ({\em e.g.}, using $k$-means~\cite{kmeans})
from preliminary weighted ensemble 
simulations with short time horizons. 
The resulting Markov state model 
can be crude: it will be used only for 
choosing parameters, and weighted ensemble is 
exact no matter the parameters. 
Indeed if our Markov state model was very 
refined, it could be used 
directly to estimate $\int f \,d\mu$. 
In practice, we expect our crude model could 
estimate $\int f \,d\mu$ with significant bias.
In our 
formulation, 
a bad set of parameters may lead to large variance, 
but there is never any bias. 
In short, the Markov state model should 
be good enough to pick reliable 
weighted ensemble parameters, 
but not necessarily good enough 
to accurately estimate $\int f \,d\mu$.

\begin{remark}\label{rmk_bins} 
We distinguish microbins from bins 
because {the number of weighted ensemble bins 
is limited by the number of 
particles we can afford to simulate 
over the time horizon needed to reach $\mu$}. 
As an extreme case, 
suppose we have many more bins 
than particles, so that 
all bins contain $0$ 
or $1$ particles at almost every time. 
Then, because Algorithm~\ref{alg1} 
requires at least one child 
per occupied bin, parents almost never
have more than one child, 
and we recover
direct Monte Carlo. This 
condition, that the collection of 
parents in 
a given bin must have at least one 
child, is essential for the 
stability of long-time calculations~\cite{aristoff2019ergodic}. 
As a result, too many bins leads 
to poor weighted ensemble performance.
A very rough rule of thumb is that 
the number $M$ of bins
should be not too much larger than the number 
$N$ of particles. 
\end{remark}

\begin{algorithm}\caption{Optimizing the bins}
Choose an initial collection ${\mathcal B}$ of bins.
Define an objective function on bin space,
\begin{equation}\label{objective}
{\mathcal O}({\mathcal B}') = \sum_{u \in {\mathcal B}'} \text{Var} (\tilde{K}\tilde{h}|_u),
\end{equation}
where $\tilde{K}\tilde{h}|_u$ is the restriction of 
$\tilde{K}\tilde{h}$ to $\{p \in {\mathcal M}{\mathcal B}: p \in u\}$, and $\text{Var} (\tilde{K}\tilde{h}|_u)$ is the usual 
vector population variance. Choose an annealing parameter $\alpha>0$, 
set ${\mathcal B}_{opt} = {\mathcal B}$, and iterate the 
following for a user-prescribed number of steps:
\vskip5pt
\begin{itemize}
\item[1.] Perturb ${\mathcal B}$ to get new bins ${\mathcal B}'$.

(Say by moving a microbin from one bin to another bin). 
\item[2.] With probability $\min\{1,\exp[\alpha({\mathcal O}(\mathcal{B}) - {\mathcal O}(\mathcal{B'}))]\}$, 
set ${\mathcal B} = {\mathcal B}'$.
\item[3.] If ${\mathcal O}({\mathcal B}) < {\mathcal O}({\mathcal B}_{opt})$, then update ${\mathcal B}_{opt} = {\mathcal B}$. Return to Step 1.
\end{itemize}
\vskip5pt
Once the bin search is complete, the output is ${\mathcal B} = {\mathcal B}_{opt}$.
\label{alg_opt_bins}
\end{algorithm} 

The number of weighted 
ensemble bins is limited 
by the number $N$ of particles. 
(In the references in the Introduction, 
$N$ is usually on the order of $10^2$ to $10^3$.)
The number of microbins, on the 
other hand, is 
limited primarily by the 
cost of the sampling 
that produces $\tilde{h}$ and $\tilde{K}$. 
The microbins could 
be computed by post-processing 
exploratory 
weighted ensemble data generated 
using larger bins. The 
{quality} of the microbin-to-microbin 
transition matrix $\tilde{K}$
depends on the microbins and the number of particles 
used in these exploratory 
simulations. 
But these exploratory simulations, 
compared to our steady-state 
weighted ensemble simulations,  
could use more particles 
as their time horizons 
can be much shorter. As a 
result the number of microbins 
can be much greater than the number of bins.

In Algorithm~\ref{alg_opt_bins}, 
we could enforce an additional 
condition 
that the bins must be connected
regions in state space. 
We do this 
in our implementation of Algorithm~\ref{alg_opt_bins} in Section~\ref{sec:numerics} below. Traditionally, bins are connected, 
not-too-elongated
regions, {\em e.g.} Voronoi 
regions. Bins are chosen this way 
because resampling in 
bins with distant particles can
lead to a large variance in  
the weighted ensemble. However, {\em since 
we are employing weighted ensemble 
only to compute $\int f\,d\mu$ for a single observable 
$f$}, a large variance 
in the full ensemble can 
be tolerated so long as 
the variance assocated with 
estimating $\int f\,d\mu$ is 
still small. This could 
be achieved with 
disconnected or elongated 
bins, or even a non-spatial 
assignment of particles to 
bins (based {\em e.g.} on the values 
of $\tilde{K}\tilde{h}$ on the 
microbins containing the particles). We leave a more complete 
investigation to future work.

\begin{algorithm}\caption{Initializing weighted ensemble}
After the preliminary simulations that produce 
a collection of particles and weights 
$(\xi_0^i,\omega^i)^{i=1,\ldots,N_{\textup{init}}}$, 
together with approximations $\tilde{K}$, 
$\tilde{\mu}$, $\tilde{f}$ and $\tilde{h}$ 
of $K$, $\mu$, $f$, and $h$: 
\vskip5pt
\begin{itemize}[leftmargin=20pt]
\item Adjust the weights of the particles 
in each microbin $p$ according to $\tilde{\mu}_p$:
\begin{equation*}
\omega_0^i = \omega^i \frac{\tilde{\mu}_p}{\sum_{j:\xi_0^j \in p}\omega^j}, \qquad \text{if }\xi_0^i \in p.
\end{equation*}
There should be at least one initial particle in each microbin, 
$$\sum_{j:\xi_0^j \in p}\omega^j>0, \qquad p \in {\mathcal M}{\mathcal B}.$$

\item Proceed to the selection step 
of Algorithm~\ref{alg1} at 
time $t = 0$, with the initial particles $\xi_0^1,\ldots,\xi_0^{N_{\textup{init}}}$ having the
adjusted weights $\omega_0^1,\ldots,\omega_0^{N_{\textup{init}}}$.
\end{itemize}
\vskip5pt
The number, $N_{\textup{init}}$, of initial particles can be much greater than the number, $N$, of particles in the weighted ensemble simulations of Algorithm~\ref{alg1}.
\label{alg_initialize}
\end{algorithm} 

\subsection{Initialization}
Note that the steady state $\tilde{\mu}$ 
of $\tilde{K}$ can be 
used to precondition the 
weighted ensemble simulations, 
so that they start closer to the true
steady state. This basically amounts 
to adjusting the weights of the
initial particles so that 
they match $\tilde{\mu}$. 
This is called {\em reweighting} 
in the weighted ensemble literature~\cite{bhatt2010steady,jeremy,suarez,zuckerman}. 
One way to do this 
is the following. Take initial 
particles and weights 
$(\xi_0^i,\omega^i)^{i=1,\ldots,N_{\textup{init}}}$ 
from the preliminary simulations 
that define 
$\tilde{K}$, $\tilde{\mu}$, $\tilde{f}$ and 
$\tilde{h}$. These initial particles can 
be a large subsample from these 
simulations; in particular 
we can have an initial number 
of particles $N_{\textup{init}}\gg N$ 
much greater than the number of 
particles in weighted ensemble 
simulations. (The large number 
can be obtained by sampling at 
different times along the
the preliminary simulation particle ancestral lines.) {\em We require that there 
is at least one initial particle in each 
microbin.} The weights $(\omega^i)^{i=1,\ldots,N_{\textup{init}}}$ of these particles 
are adjusted using $\tilde{\mu}$ to 
get new weights $(\omega_0^i)^{i=1,\ldots,N_{\textup{init}}}$, such 
that the total adjusted weight in 
each microbin matches the value 
of $\tilde{\mu}$ on the same microbin.
Then these $N_{\textup{init}}\gg N$ 
initial particles are fed into the first ($t = 0$) 
selection step of Algorithm~\ref{alg1}.
This selection step prunes the number 
of particles to a manageable 
number, $N$, for the rest of the 
main weighted ensemble simulation.
See Algorithm~\ref{alg_initialize} 
for a precise description of this 
initialization.

\subsection{Gain over naive parameter choices}

The gain of 
optimizing parameters, compared 
to naive parameter choices 
or direct Monte Carlo, comes 
from the larger number of particles 
that optimized weighted ensemble 
puts in important regions 
of state space, compared to 
a naive method. {\em These 
important regions are exactly 
the ones identified by $h$}; 
roughly speaking they are 
regions $R$ where the 
variance of $h$ is large. 
A rule of thumb is 
that the variance can decrease by a 
factor of up to 
\begin{equation}\label{gain}
\frac{\text{average }\#\text{ of particles in }R \text{ with optimized parameters}}{\text{average }\#\text{ of particles in }R \text{ with naive method}}.
\end{equation}

To see why, consider the mutation variance from 
Proposition~\ref{prop_vis_var},
\begin{align}\begin{split}\label{vis_mut_var2}
\lim_{T \to \infty} T^2{\mathbb E}\left[\left.\left(D_{t+1}-\hat{D}_{t}\right)^2\right|{\mathcal F}_{t}\right] 
=    \sum_{u \in {\mathcal B}}\frac{\omega_t(u)}{N_t(u)} \sum_{i: {\xi}_t^i \in u}\omega_t^i\textup{Var}_{K(\xi_t^i,\cdot)}h.
\end{split}
\end{align}
The contribution to this mutation variance at time $t$ from bin $u$ is 
\begin{equation*}
\frac{\omega_t(u)}{N_t(u)} \sum_{i: {\xi}_t^i \in u}\omega_t^i\textup{Var}_{K(\xi_t^i,\cdot)}h.
\end{equation*}
Increasing $N_t(u)$ by some 
factor decreases the mutation variance from bin $u$ at 
time~$t$ by the same factor. 
The mutation variance can be 
reduced by a factor of almost~\eqref{gain}  
if $N_t(u)$ is 
increased by the  
factor~\eqref{gain} in the 
bins $u$ where $\textup{Var}_{K(\xi_t^i,\cdot)}h$ 
is large,
and if $\textup{Var}_{K(\xi_t^i,\cdot)}h$ is small enough in the other 
bins that decreasing the allocation in those bins
does not significantly 
increase the mutation variance.

Of course 
the variance formulas in 
Lemmas~\ref{lem_mut_var} and 
Lemma~\ref{lem_sel_var} can, 
in principle, 
more precisely describe 
the gain, although it is 
difficult to accurately estimate the 
values of these variances a priori
outside of the $N \to \infty$ 
limit. Since we focus on relatively small $N$, 
we do not go in this analytic direction.
Instead we numerically illustrate 
the improvement from optimizing 
parameters in Figures~\ref{fig_vary_bins} and~\ref{fig_optimization_data}.

\subsection{Adaptive methods}

We have proposed handling the 
optimizations~\eqref{opt_allocation} and~\eqref{opt_bins} 
by approximating $K$ and $h$ with 
${\tilde{K}}$ and $\tilde{h}$, where 
the latter are built by observing microbin-to-microbin 
transitions and solving the appropriate 
Poisson problem. Since $K$ and $h$ are fixed in time 
and we are doing long-time calculations, 
it is natural to estimate them adaptively, for instance 
via stochastic approximation~\cite{kushner2003stochastic}.
Thus both~\eqref{opt_allocation} 
and~\eqref{opt_bins} could be solved 
adaptively, at least in principle.
Depending 
on the number of microbins, it 
may be relatively cheap to compute  
$h$ compared to the cost of evolving 
particles. If this is the case,
it is natural to solve~\eqref{opt_allocation} 
on the fly. We could also perform 
bin searches intermittently, depending 
on their cost.

\section{Numerical illustration}\label{sec:numerics}
In this section we illustrate 
the optimizations in 
Algorithms~\ref{alg_opt_allocation} 
and~\ref{alg_opt_bins}, for a simple example of 
a mean first passage time computation. Consider 
{\em overdamped Langevin dynamics} 
\begin{equation}\label{ovd_lang}
dX_t = -V'(X_t)\,dt + \sqrt{2\beta^{-1}}dW_t,
\end{equation}
where $\beta = 5$, $(W_t)_{t \ge 0}$ is a standard Brownian motion, and the potential energy is 
\begin{equation*}
V(x) = \begin{cases} 5(x-7/12)^2 + 0.15\cos(240\pi x), & x < 7/12\\
-1 - \cos(12\pi x) + 0.15 \cos(240 \pi x), & x \ge 7/12\end{cases}.
\end{equation*}
See Figure~\ref{fig_V_h_pi}. We impose reflecting boundary conditions on the interval $[0,1]$. 
We will estimate the mean first passage time 
of $(X_t)_{t \ge 0}$ from $1/2$ to $[119/120,1]$ 
using the Hill relation~\cite{aristoff2016analysis,hill2005free}. 
The Hill relation reformulates the 
mean first passage time as a steady-state 
average, as we explain below.

\begin{figure}
\includegraphics[width=13cm]
{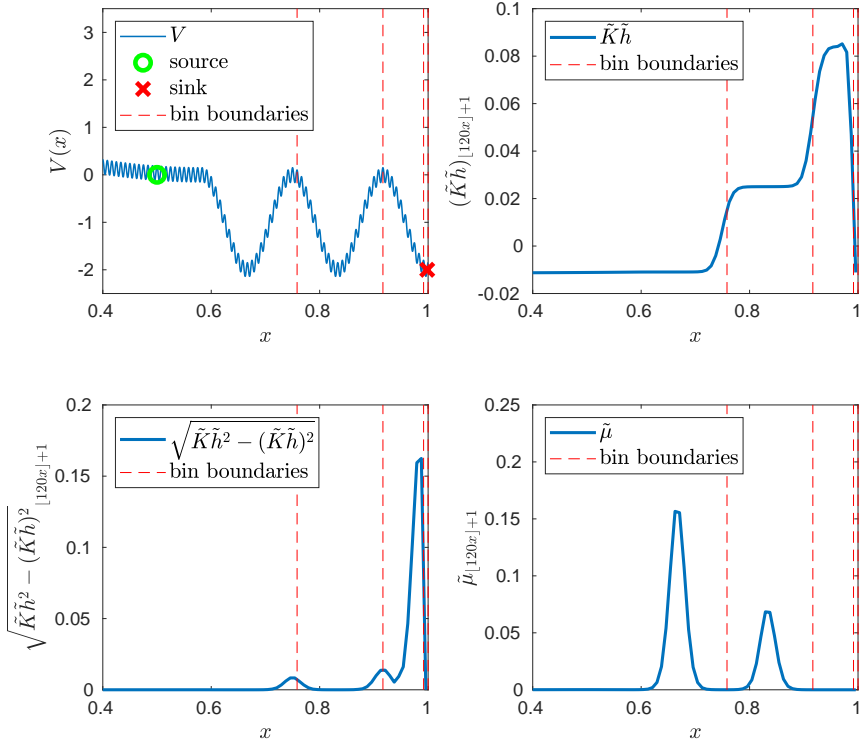}
\vskip-10pt
\caption{Using 
Algorithm~\ref{alg_opt_bins} 
to compute the weighted ensemble bins ${\mathcal B}$ 
when the number of bins is $M = 4$. We use $10^6$ iterations 
of Algorithm~\ref{alg_opt_bins} with 
$\alpha = 10^5$. Top left: Potential energy $V$ and bin boundaries when $M = 4$. 
Top right: The vector ${\tilde K}{\tilde h}$ defining the objective function in Algorithm~\ref{alg_opt_bins}, where $\tilde{h}$ is 
the approximate Poisson solution. Note 
that ${\tilde K}{\tilde h}$ is nearly 
constant on each superbasin. 
Bottom left: (square root of) the vector $\tilde{K}\tilde{h}^2 - (\tilde{K}\tilde{h})^2$ involved in the mutation 
variance optimization in 
Algorithm~\ref{alg_opt_allocation}. Bottom right: Approximate steady state 
distribution $\tilde{\mu}$.  All plots have been cropped at $x>0.4$, where the values of  
$\tilde{K}\tilde{h}^2 - (\tilde{K}\tilde{h})^2$ and $\tilde{\mu}$ 
are neglibigle and ${\tilde K}{\tilde h}$ is nearly constant.}
\label{fig_V_h_pi}
\end{figure}

We choose $120$ uniformly spaced 
microbins between $0$ and $1$,
$${\mathcal M}{\mathcal B} = \{[(p-1)/120,p/120]:i=1,\ldots,120\}.$$
(They are not actually disjoint, but they do not overlap so this is unimportant.) 

The microbins correspond to 
the {\em basins of attraction of $V$}.
A basin of attraction of $V$ is a 
set of initial conditions $x(0)$ for which
$dx(t)/dt = - V'(x(t))$ has a unique long-time limit. The microbins 
comprise $3$ larger 
{\em metastable sets}, defined 
in Section~\ref{sec:notation} above. These larger metastable 
sets, each comprised of many smaller 
basins of attraction, will be 
called {\em superbasins}. 
{\em The microbins do not need to be basins of 
attraction: they only need to 
be sufficiently ``small''
to give useful estimates of $K$ 
and $h$.} 
We choose microbins in this 
way to illustrate the qualitative 
features we expect from 
Markov state modeling, where the 
Markov states (or our microbins) 
are often basins 
of attraction. In 
this case, the bins 
might be clusters of Markov 
states corresponding 
to superbasins.

\begin{figure}
\includegraphics[width=13cm]
{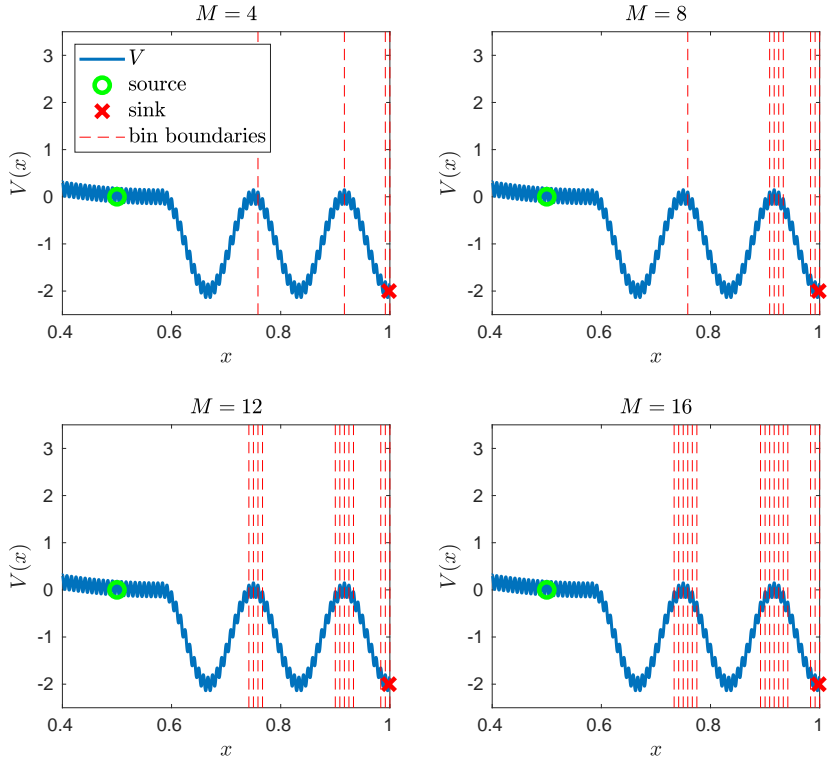}
\vskip-10pt
\caption{Increasing the value of $M$ 
in Algorithm~\ref{alg_opt_bins}. Plotted 
are weighted ensemble bins ${\mathcal B}$ computed 
using Algorithm~\ref{alg_opt_bins} with 
$M = 4,8,12,16$. For each 
value of $M$, we use $10^6$ iterations 
of Algorithm~\ref{alg_opt_bins}, with 
$\alpha$ tuned between $10^5$ and $10^6$. 
Note that with increasing 
$M$, additional bins are initially devoted 
to resolving the energy barrier 
between the two rightmost superbasins. 
Since the observable $f$ is the 
rightmost microbin, this is the 
most important energy barrier for 
the bins to resolve. 
Note that the multiple adjacent small 
bins for $M = 16$ correspond to the steepest gradients of $\tilde{K}\tilde{h}$
in the top right of Figure~\ref{fig_V_h_pi}. All plots 
have been cropped at $x>0.4$.}
\label{fig_bin_boundaries}
\end{figure}

\begin{figure}
\centering
\includegraphics[width=13cm]
{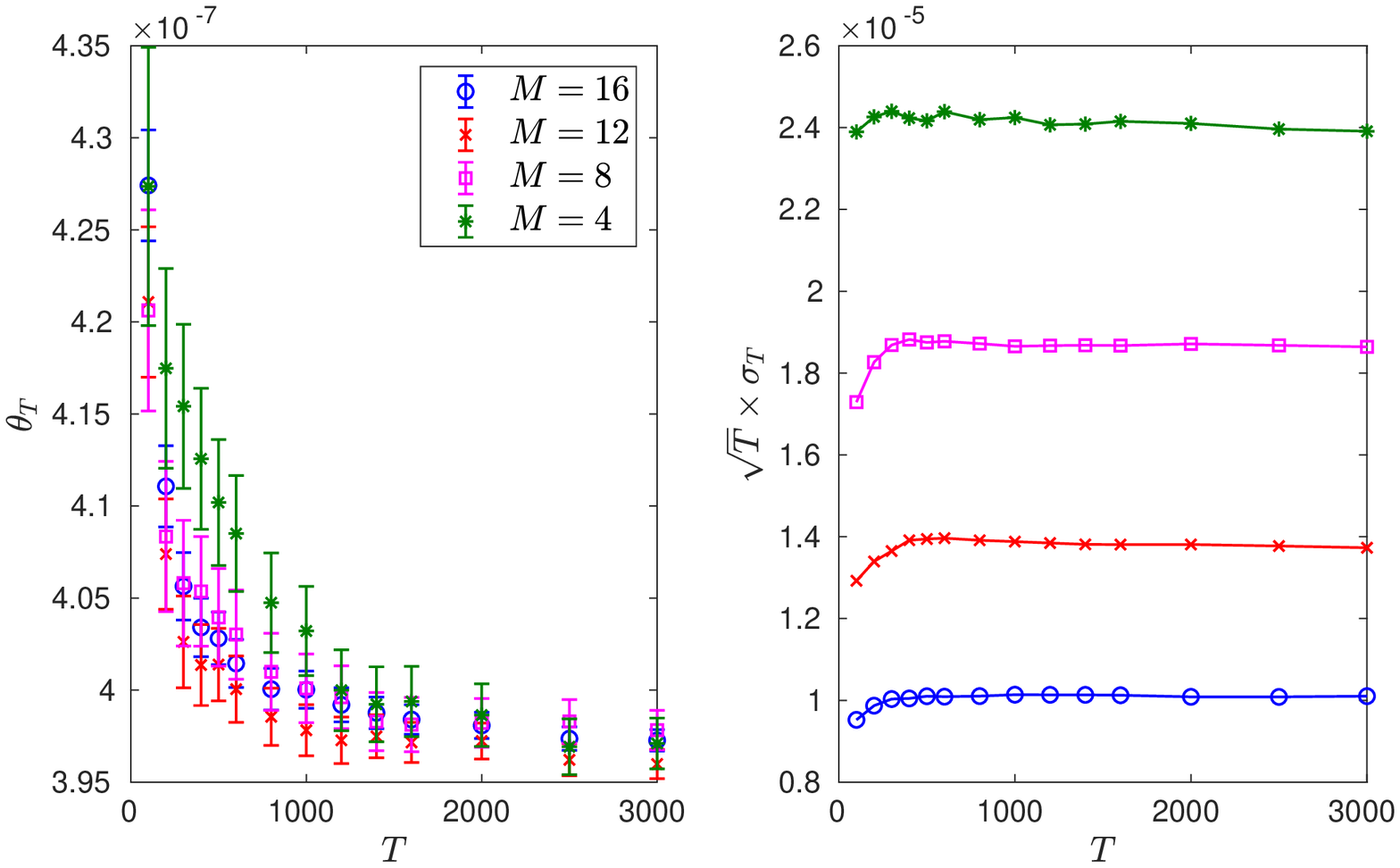}
\vskip-10pt
\caption{Varying the number $M$ 
of weighted ensemble bins 
in Algorithm~\ref{alg_opt_bins}.
Left: Weighted ensemble running means $\theta_T$ 
vs. $T$ for Algorithm~\ref{alg1} 
with the optimal allocation 
and binning of Algorithms~\ref{alg_opt_allocation} and~\ref{alg_opt_bins}, when $M = 4,8,12,16$. Values 
shown are averages over 
$10^5$ independent 
trials. Error bars 
are $\sigma_T/\sqrt{10^5}$, where 
$\sigma_T$ are the empirical 
standard deviations. Right: 
Scaled empirical standard 
deviations $\sqrt{T}\times \sigma_T$ 
vs. $T$ in the same setup. 
We use $N = 40$ particles, and the
bins correspond 
exactly to the ones in Figure~\ref{fig_bin_boundaries}. More bins 
is not always better, since with 
too many bins we return to the 
direct Monte Carlo regime; 
see Remark~\ref{rmk_bins}.}
\label{fig_vary_bins}
\end{figure}

The kernel $K$ is defined as follows. 
First, let $K_{\delta t}$ be an 
Euler-Maruyama time discretization of~\eqref{ovd_lang} with time step 
$\delta t = 2 \times 10^{-5}$. 
We introduce a {\em sink} at the target 
state $[119/120,1]$ that recycles at 
a {\em source} $x = 1/2$ via 
\begin{equation*}
\bar{K}_{\delta t}(x,dy) = \begin{cases} K_{\delta t}(\frac{1}{2},dy), & x \in [119/120,1] \\
K_{\delta t}(x,dy), & x \notin [119/120,1]\end{cases}.
\end{equation*}
Then we define $K$ as a $\Delta t$-skeleton of $\bar{K}_{\delta t}$ where $\Delta t = 10$, 
\begin{equation}\label{between}
K = \bar{K}_{\delta t}^{10}.
\end{equation}

This just means we take $10$ Euler-Maruyama~\cite{kloeden} time steps in the mutation 
step of Algorithm~\ref{alg1}. The 
Hill relation~\cite{aristoff2016analysis} 
shows that, if $(\bar{X}_n)_{n \ge 0}$ is a Markov chain with kernel either ${K}_{\delta t}$ or $\bar{K}_{\delta t}$ and $\bar{\tau} = \inf\{n\ge 0:\bar{X}_n \in [119/120,1]\}$, then 
\begin{equation*}
{\mathbb E}[\bar{\tau}|\bar{X}_0 = 1/2] = \frac{1}{\mu([119/120,1])},
\end{equation*}
where $\mu$ is the stationary distribution of $K$. Thus if $\tau = \inf\{t>0:X_t \in [119/120,1]\}$, 
\begin{equation*}
{\mathbb E}[\tau|X_0 = 1/2] \approx \frac{ \delta t}{\mu([119/120,1])}.
\end{equation*}
By construction $\mu([119/120,1])$ is small 
(on the order $10^{-7}$), so it must be estimated with 
substantial precision to 
resolve the mean first passage 
time.

We will estimate the 
mean first passage time 
from $x = 1/2$ to $x \in [119/120,1]$, 
the latter being the target 
state and rightmost 
microbin. Thus, we define $f$ 
as the characteristic function of 
this microbin,
$f = \mathbbm{1}_{[119/200,1]}$, so that
$$\theta_T \approx \int f\,d\mu = \mu([119/120,1]).$$
Weighted ensemble then estimates the mean first passage time via 
\begin{equation}\label{MFPT2}
\textup{Mean first passage time } = {\mathbb E}[\tau|X_0 = 1/2] \approx \frac{\delta t}{\theta_T}.
\end{equation}

We illustrate Algorithm~\ref{alg1} combined 
with Algorithms~\ref{alg_opt_allocation}-\ref{alg_opt_bins} in Figures~\ref{fig_V_h_pi}-\ref{fig_optimization_data}. For Algorithms~\ref{alg_opt_allocation} and~\ref{alg_opt_bins}, 
to construct $\tilde{K}$, $\tilde \mu$, 
$\tilde{f}$ and $\tilde{h}$ as in 
Section~\ref{sec:params}, we compute the matrix 
$\tilde{K}$ using $10^4$ trajectories 
starting from each microbin's midpoint, 
and we define $\tilde{f}_{120} = 1$, $\tilde{f}_p = 0$ for $1 \le p \le 119$. The $p$th 
rows of $\tilde{K}$, $\tilde \mu$, 
$\tilde{f}$ and $\tilde{h}$ correspond 
to the microbin $[(p-1)/120,p/120]$. 
In Algorithm~\ref{alg_opt_bins}, 
to simplify visualization 
we enforce a condition that 
the bins must be connected 
regions. We initialize all our simulations 
with Algorithm~\ref{alg_initialize} where $N_{\textup{init}} = 120$,
$\xi_0^i = (i-1/2)/120$, and $\omega^i = 1/120$.

In Figure~\ref{fig_V_h_pi}, we show the bins 
resulting from Algorithm~\ref{alg_opt_bins}, 
and plot the terms $\tilde{K}\tilde{h}$ and
$\tilde{K}\tilde{h}^2 - (\tilde{K}\tilde{h})^2$ 
appearing in the optimizations in 
Algorithms~\ref{alg_opt_allocation} 
and~\ref{alg_opt_bins}, along with 
the approximate steady state $\tilde{\mu}$. Note that 
the bins
resolve the superbasins of $V$. 
In Figure~\ref{fig_bin_boundaries}, we 
explore what happens when the number 
$M$ of bins increases in Algorithm~\ref{alg_opt_bins}, finding that 
the bins begin to resolve the 
regions between superbasins, 
favoring regions closer to the support 
of $f$. As $M$ increases, 
the optimal allocation 
leads to particles having 
more children when they are near 
the dividing surfaces between the superbasins; to see why, compare the particle allocation 
in
Algorithm~\ref{alg_opt_allocation} with
the bottom left of Figure~\ref{fig_V_h_pi}.

In Figure~\ref{fig_vary_bins}, we illustrate 
weighted ensemble with the optimized 
allocation and binning 
of Algorithms~\ref{alg_opt_allocation} 
and~\ref{alg_opt_bins} when 
the number of bins increases 
from $M = 4$ to $M = 16$. 
Observe that $M = 4$ bins is not enough 
to resolve the regions between the superbasins, 
but we still get a substantial gain over 
direct Monte Carlo (compare with Figure~\ref{fig_optimization_data}). With $M = 16$ bins 
we resolve the regions between superbasins,
further reducing the variance. 
The bins we use are exactly the ones in Figure~\ref{fig_bin_boundaries}. 

\begin{figure}
\includegraphics[width=13cm]
{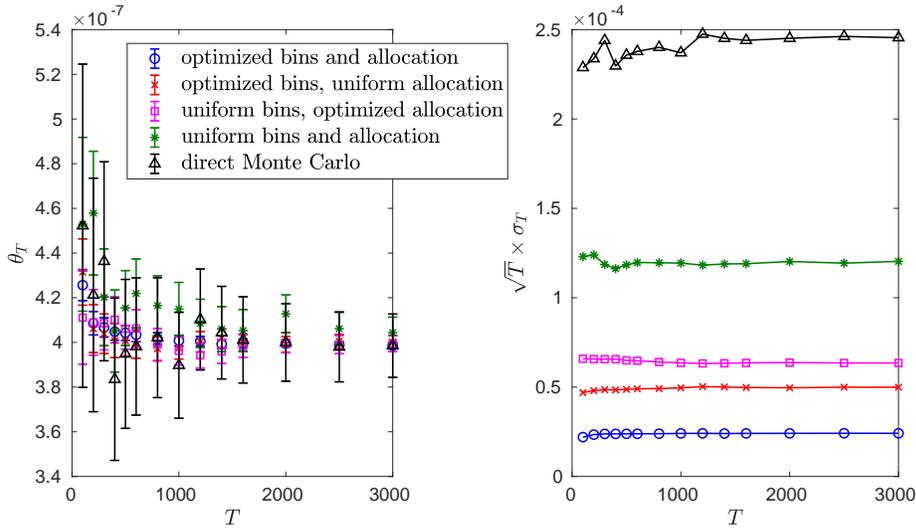}
\vskip-10pt
\caption{Comparison of direct Monte 
Carlo with weighted ensemble, where 
the bins and/or allocation are 
optimized, $M = 4$, and $N = 40$. Left: 
Weighted ensemble running means $\theta_T$ 
vs. $T$ for Algorithm~\ref{alg1} 
with the indicated 
choices for allocation 
and binning. Values 
shown are averages over 
$10^5$ independent 
trials. Error bars 
are $\sigma_T/\sqrt{10^5}$, where 
$\sigma_T$ are the empirical 
standard deviations. Right: 
Scaled empirical standard 
deviations $\sqrt{T}\times \sigma_T$ 
vs. $T$ in the same setup.
By optimized bins, 
we mean the weighted ensemble 
bins 
${\mathcal B}$ are chosen using 
Algorithm~\ref{alg_opt_bins} 
with $M = 4$. These optimized bins are 
exactly the ones plotted 
in top left of Figure~\ref{fig_bin_boundaries}.
Optimized allocation means 
the particle allocation follows 
Algorithm~\ref{alg_opt_allocation}. 
Uniform bins means that 
the bins are uniformly spaced on $[0,1]$, while 
uniform allocation 
means the particles are distributed 
uniformly among the occupied bins.}
\label{fig_optimization_data}
\end{figure}

In Figure~\ref{fig_optimization_data}, 
we compare weighted ensemble with 
direct Monte Carlo when $M = 4$. Direct 
Monte Carlo can be seen as a 
version of Algorithm~\ref{alg1} where
each parent always has exactly one child,
$C_t^i = 1$ for all $t,i$. For weighted ensemble, 
we consider optimizing either or 
both of the bins and the allocation. 
When we do not optimize the bins, 
we consider {\em uniform bins} 
${\mathcal B} = \{[0,40/120],[40/120,80/120],[80/120,1]\}.$ 
When we do not optimize the allocation, 
we consider {\em uniform allocation}, 
where we distribute particles evenly 
among the occupied bins, 
$N_t(u) \approx N/\#\{u\in {\mathcal B}:\omega_t(u)>0\}$.  Notice the order of magnitude 
reduction in standard deviation compared 
with direct Monte Carlo for this 
relatively small number of bins.

In 
Figure~\ref{fig_MFPT}, to 
illustrate the Hill relation, we 
plot the weighted ensemble 
estimates of the mean first 
passage time from our data in Figures~\ref{fig_vary_bins} 
and~\ref{fig_optimization_data} 
against the (numerically) exact value. 
The weighted ensemble estimates at small $T$ tend to 
exhibit a ``bias'' because the weighted 
ensemble has not yet reached steady state. 
As $T$ grows, this bias vanishes and 
the weighted ensemble estimates converge 
to the true mean first passage time. 
In the simple example in this section, we can 
directly compute the mean first passage time. For complicated biological problems, 
however, the 
first passage time can 
be so large that it is difficult to 
directly sample even once. 
In spite of this, the Hill relation reformulation
can lead to useful estimates on much 
shorter time scales than 
the mean first passage time \cite{adhikari2019computational,jeremy2,danonline}. 
Indeed this is the case in Figure~\ref{fig_MFPT}, 
where we get accurate estimates 
in a time horizon orders of magnitude 
smaller than 
the mean first passage time. 
This speedup 
can be attributed in part to 
the initialization in Algorithm~\ref{alg1}. 
In general, there can also be 
substantial speedup from the Hill relation itself, 
independently of the initial 
condition~\cite{danonline}.

Optimizing 
the bins and allocation together 
has the best performance. 
We expect that, as in this 
numerical example,
when the number $M$ of bins is 
relatively small,
optimizing the bins 
can be more important than optimizing 
the allocation.  
Optimizing only the allocation 
may lead to a less dramatic 
gain when the bins poorly resolve 
the landscape defined by $h$. 
On the other hand, for a large enough number 
$M$ of bins it may be sufficient 
just to optimize the allocation. 
We emphasize that {\em more bins 
is not always better}, 
since for a fixed number $N$ of 
particles, with too many bins we 
end up recovering direct Monte Carlo. See Remark~\ref{rmk_bins} 
above.

\begin{figure}
\includegraphics[width=13cm]
{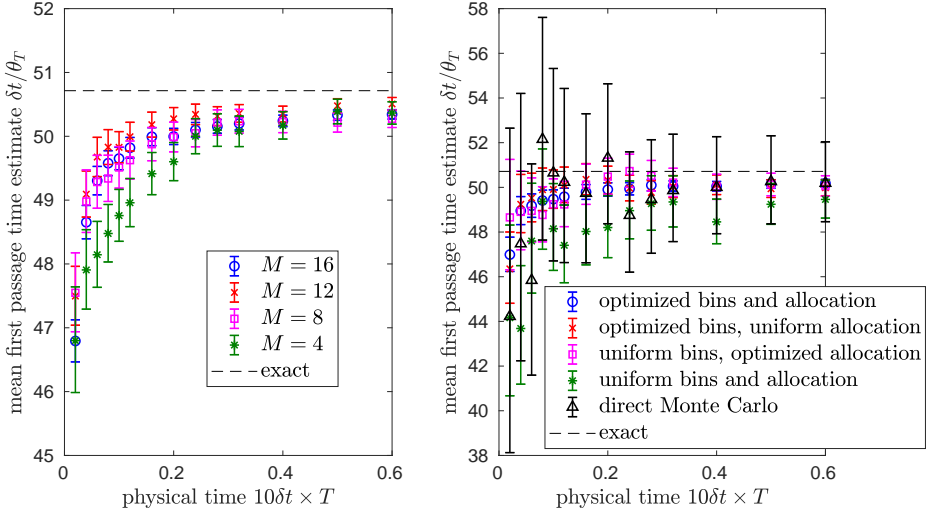}
\vskip-10pt
\caption{Illustration of the Hill relation for estimating the mean first passage time via~\eqref{MFPT2}. 
Left: the same data as in Figure~\ref{fig_vary_bins}, but $\theta_T$ is inverted and multiplied 
by $\delta t$ to estimate the 
mean first passage time, and error 
bars are adjusted accordingly.
Right: the same data as in Figure~\ref{fig_optimization_data}, but $\theta_T$ 
is inverted and multiplied 
by $\delta t$ to estimate the 
mean first passage time, and error 
bars are adjusted accordingly. 
So that the $x$ and $y$ axes have 
the same units, we consider 
the ``physical time'' on the $x$-axis, 
defined as the total number, $T$, of 
selection steps of Algorithm~\ref{alg1}
multiplied by the time, $10\delta t$, between 
selection steps (see~\eqref{between}). In 
both plots, the ``exact'' value 
of the mean first passage time 
is obtained from $10^5$ independent 
samples of the first passage 
time $\bar{\tau}$. As in Figures~\ref{fig_vary_bins} and~\ref{fig_optimization_data}, 
optimizing the bins and allocation 
has the best performance, 
and $M = 16$ bins is better than 
$M=4,8,12$ bins, though more 
bins is not always better; see Remark~\ref{rmk_bins}.
}
\label{fig_MFPT}
\end{figure}

\section{Remarks and future work}\label{sec:remarks}

This work presents new procedures for 
choosing the bins and particle 
allocation for weighted ensemble, 
building in part from ideas in~\cite{aristoff2016analysis}. 
The bins and particle allocation, 
together with the resampling 
times, completely characterize 
the method. Though we do not 
try to optimize the latter, we argue that 
there is no significant variance 
cost associated with taking 
small resampling times. 
Optimizing weighted 
ensemble is worthwhile when 
the optimized parameter choices 
lead to significantly more particles 
in important regions of state 
space, compared to a naive method. 
The corresponding gain, represented by the 
rule of thumb~\eqref{gain}, can 
be expected to grow as the dimension 
increases. This is because in high dimension 
the pathways between metastable sets are 
narrow compared to the 
vast state space (see~\cite{vanden2005transition,weinan2010transition} and 
the references in the Introduction).

Though our interest is 
in steady state calculations, 
many of our ideas could just 
as well be applied to a finite 
time setup. Our practical interest in 
steady state arises from the 
computation of mean first 
passage times via the Hill 
relation. On the mathematical 
side, general 
importance sampling
particle methods with 
the unbiased property of 
Theorem~\ref{thm_unbiased} 
are often not appropriate for 
steady state sampling. 
(By general methods, we 
mean methods for {\em nonreversible} Markov 
chains, or Markov chains 
with a steady state that is {\em not 
known up to a normalization 
factor}~\cite{lelievre2016partial}.) Indeed as we explain 
in our companion article~\cite{aristoff2019ergodic}, 
standard sequential Monte 
Carlo methods can suffer from 
variance explosion at large times,
even for carefully designed 
resampling steps. 
So this article could be 
an important 
advance in that direction.

We conclude by discussing some 
open problems. The implementation of 
the algorithms in this article 
to complex, high dimensional 
problems will require substantial 
effort and 
modification to the software in~\cite{westpa}.
The aim of this article is to 
lay the groundwork for such 
applications. On a more theoretical 
note, there remain some questions 
regarding parameter optimization. 
We could consider adaptive/non-spatial 
bins instead of fixed 
bins as in Algorithm~\ref{alg_opt_bins}, 
for instance bins chosen via $k$-means 
on the values of $Kh$ of the $N$ particles.
Also,  
we choose a number $M$ of weighted 
ensemble bins and number $N$ 
of particles {\em a priori}. It remains open 
how to pick the best value of 
$M$ for a given number, $N$, of 
particles, and fixed computational 
resources. Using only the variance 
analysis above, a straightforward 
answer to this question 
can probably only be obtained in 
the large $N$ asymptotic regime. 
More analysis is then needed, 
since we generally have 
small $N$.
We could also optimize over 
both $N$ 
and $M$, or over $N$, 
$M$ and a total number $S$
of independent simulations, 
subject to a fixed
computational budget. We leave 
these questions to future work.

\section*{Acknowledgements}

D. Aristoff thanks Peter Christman, 
Tony Leli{\`e}vre, Josselin Garnier, Matthias Rousset,  
Gabriel Stoltz, and Robert J. Webber for helpful suggestions and 
discussions. D. Aristoff and 
D.M. Zuckerman especially thank 
Jeremy Copperman and Gideon Simpson 
for many ongoing discussions related 
to this work.

This work was supported by the National Science Foundation via the awards DMS-1818726 and DMS-1522398 
(for author 
D. Aristoff)
and by the National Institutes of Health via the award GM115805 (for author D.M. Zuckerman).

\section{Appendix}\label{sec:appendix}

In this appendix, we describe residual 
resampling, and we draw a connection 
between mutation variance minimization 
and sensitivity.

\subsection{Residual resampling}

Recall that in Algorithm~\ref{alg1} we 
assumed residual resampling is used 
to get the number 
of children $(C_t^j)^{j=1,\ldots,N}$
of each particle $\xi_t^j$ at each time $t$. 
We could also use residual resampling to 
compute $R_t(u)^{u \in {\mathcal B}}$ in 
Algorithm~\ref{alg_opt_allocation}.
For the reader's convenience, we describe 
residual resampling in Algorithm~\ref{alg_residual}. Recall $\lfloor x \rfloor$ is the floor function (the greatest integer $\le x$).

\begin{algorithm}\caption{Residual resampling}

To generate $n$ samples $\{N_i:i \in {\mathcal I}\}$ from a distribution $\{d_i: i \in {\mathcal I}\}$ with $\sum_{i \in {\mathcal I}}d_i = 1$:
\begin{itemize}[leftmargin=20pt]
\item Define $\delta_i = nd_i - \lfloor nd_i \rfloor$ and let $\delta = \sum_{i \in {\mathcal I}}\delta_i$.
\item Sample 
$\{R_i: i \in {\mathcal I}\}$ from the multinomial 
distribution 
with $\delta$ trials and event probabilities 
$\delta_i/\delta$, $i \in {\mathcal I}$.
In detail, 
\begin{align*}
&{\mathbb P}(R_i = r_i, \,i\in {\mathcal I}) =
\mathbbm{1}_{\sum_{i \in {\mathcal I}} r_i = \delta} \frac{\delta!}{\prod_{i \in {\mathcal I}} r_i!}\prod_{i \in {\mathcal I}} (\delta_i/\delta)^{r_i}.
\end{align*}
\item Define $
N_i = \lfloor nd_i\rfloor + R_i$, $i \in {\mathcal I}$.
\end{itemize}
\label{alg_residual}
\end{algorithm}

\subsection{Intuition behind mutation variance minimization}\label{sec:int_mut}

The goal of this section is to give some intuition 
for the strategy~\eqref{Ntu}, which chooses 
the particle allocation to minimize 
mutation variance, by introducing a 
connection with sensitivity of the stationary 
distribution $\mu$ to perturbations of $K$. {\em In this section we make the 
simplifying assumption that state 
space is finite, and each point of 
state space is a microbin.}
Thus $K = (K_{pq})_{p,q \in {\mathcal M}{\mathcal B}}$ is a 
finite stochastic matrix. The 
Poisson solution $h$ defined in~\eqref{h}
satisfies $(I-K)h = f - (f^T \mu)\mathbbm{1}$ and $h^T \mu = 0$, or equivalently
\begin{equation}\label{poiss}
h = \sum_{s=0}^\infty K^s \bar{f}, \qquad \bar{f} = f - (f^T \mu)\mathbbm{1}.
\end{equation}
where $h = (h_p)_{p \in {\mathcal M}{\mathcal B}}$, $f = (f_p)_{p \in {\mathcal M}{\mathcal B}}$, $\bar{f} = (\bar{f}_p)_{p \in {\mathcal M}{\mathcal B}}$, and $\mu = (\mu_p)_{p \in {\mathcal M}{\mathcal B}}$ are column 
vectors, and $\mathbbm{1} = \sum_{p \in {\mathcal M}{\mathcal B}} e_p$ is the all ones column vector. Here, $I$ is the identity matrix, and $(e_p)_{p \in {\mathcal M}{\mathcal B}}$ is the column vector with $1$ in the $p$th entry and $0$'s elsewhere. 

We write $\mu(Q)$ for the stationary 
distribution of an irreducible 
stochastic matrix $Q = (Q_{pq})_{p,q \in {\mathcal M}{\mathcal B}}$. More 
precisely, $\mu(Q)$ denotes a continuously differentiable 
extension of $Q \mapsto \mu(Q)$ to 
an open neighborhood of the space 
of $\#{\mathcal M}{\mathcal B}\times \#{\mathcal M}{\mathcal B}$ irreducible stochastic 
matrices which satisfies $\mu(Q)^T Q = \mu(Q)^T$ 
and $\mu(Q)^T \mathbbm{1} = 1$ whenever $Q\mathbbm{1} = \mathbbm{1}$. See~\cite{thiede2015sharp} 
for details and a proof of existence of this extension.  
Abusing notation, we still write $\mu$ 
with no matrix argument to denote the stationary distribution $\mu(K)$ of $K$. 

\begin{theorem}\label{thm_pert}
Let $\lambda(u)>0$ 
be such that $\sum_{u\in {\mathcal B}}\lambda(u) = 1$. For each $u \in {\mathcal B}$, let $\nu(u) = (\nu(u)_p)_{p \in {\mathcal M}{\mathcal B}}$ be a vector satisfying 
$\nu(u)_p \ge 0$ for $p \in {\mathcal M}{\mathcal B}$, $\nu(u)_p = 0$ for $p \notin u$, and $\sum_{p \in u}\nu(u)_p = 1$. Let $A(u)$ be a random 
matrix with the distribution 
\begin{equation}\label{dist_A}
{\mathbb P}\left(A(u) =  e_p e_q^T- e_pe_p^T K\right) = \nu(u)_pK_{pq}, \qquad \text{if }p \in u, \,q \in {\mathcal M}{\mathcal B}.
\end{equation}
Let $A(u)^{(n)}$ be independent copies of $A$, for $u \in {\mathcal B}$ and $n = 1,2,\ldots$. Define 
\begin{equation}\label{BN}
B^{(N)} = \sqrt{N}\sum_{u \in {\mathcal B}} \frac{1}{\lfloor N \lambda(u)\rfloor}\sum_{n=1}^{\lfloor N\lambda(u)\rfloor} A(u)^{(n)}.
\end{equation}
Then 
\begin{align}\begin{split}\label{sensitivity}
&\lim_{N \to \infty}  {\mathbb E}\left[\left(\left.\frac{d}{d\epsilon}\mu(K+\epsilon B^{(N)})^T\right|_{\epsilon = 0}f \right)^2\right]  \\
&\qquad=  \sum_{u \in {\mathcal B}}\lambda(u)^{-1}\sum_{p \in u} \nu(u)_p\mu_p^2\left[(Kh^2)_p - (Kh)_p^2\right].
\end{split}
\end{align}
\end{theorem}

We now interpret Theorem~\ref{thm_pert} 
from the point of view of particle 
allocation. Note that $A(u)$ 
is a centered sample of
of $K$ obtained by picking 
a microbin $p \in u$ according to the 
distribution $\nu(u)$, 
and then simulating a transition 
from $p$ via $K$. By a centered sample, 
we mean that we adjust $A(u)$ by 
subtracting by its mean, so that $A(u)$ has mean zero. 
The mean of $[(d/d\epsilon)\mu(K+\epsilon A(u))^T|_{\epsilon = 0}f]^2$ measures the sensitivity of $\int f\,d\mu = f^T\mu$ 
to sampling from bin $u$ according to the 
distribution $\nu(u)$. Similarly, the mean of $[(d/d\epsilon)\mu(K+\epsilon B^{(N)})^T|_{\epsilon = 0}f]^2$ 
measures the sensitivity of $\int f\,d\mu$ 
corresponding to sampling from each 
bin $u \in {\mathcal B}$ exactly
$\lfloor N\lambda(u)\rfloor$ times according 
to the distributions $\nu(u)$. Appropriately 
normalizing the latter in the 
limit $N \to \infty$ leads to~\eqref{sensitivity}.
The 
sensitivity in~\eqref{sensitivity} is minimized over $\lambda(u)^{u \in {\mathcal B}}$ when
\begin{equation}\label{Ntu2}
N\lambda(u) = \dfrac{N\sqrt{\sum_{p \in u} \nu(u)_p\mu_p^2\left[(Kh^2)_p - (Kh)_p^2\right]}}{\sum_{u \in {\mathcal B}}\sqrt{\sum_{p \in u} \nu(u)_p\mu_p^2\left[(Kh^2)_p - (Kh)_p^2\right]}}.
\end{equation}
 
In light of the discussion above, we think of $N\lambda(u)^{u \in {\mathcal B}}$ in~\eqref{Ntu2} as a particle 
allocation. Note that this resembles 
our formula~\eqref{Ntu} for the optimal 
weighted ensemble particle allocation.
We now try to make a connection between~\eqref{Ntu} and~\eqref{Ntu2}.  

We 
first consider a simple case. 
Suppose the bins are equal to the microbins, ${\mathcal M}{\mathcal B} = {\mathcal B}$. 
Since we assume here that every point in 
state space is a microbin, this means that every point of state space 
is also a bin. In particular the 
distributions $\nu(u)$ are trivial: $\nu(u)_p =1$ whenever $p \in {\mathcal M}{\mathcal B}$ with $p = u$. To 
make~\eqref{Ntu2} agree with~\eqref{Ntu}, 
we need $\mu_p = \omega_t(u)$ when 
$p =u$. Of 
course this equality does not hold. 
It is true that  
$\mu_p \approx \omega_t(u)$ when $p=u$ is 
a reasonable approximation, 
{\em but only in the asymptotic where $N,t \to \infty$.}
To see why this is so, 
note that the 
unbiased property and ergodicity 
show that $\lim_{t \to \infty} {\mathbb E}[\omega_t(u)] = \mu_p$ when $p=u$; 
provided an appropriate law 
of large numbers also holds, $\lim_{t \to \infty} \lim_{N \to \infty} \omega_t(u) = \mu_p$. 
Recall, though, that 
we are interested in relatively 
small $N$, due to the high cost of 
particle evolution.

For the general case, 
where ${\mathcal M}{\mathcal B} \ne {\mathcal B}$ and bins contain multiple points in state space, we see no direct connection between the allocation 
formulas~\eqref{Ntu2} 
and~\eqref{Ntu}. Indeed, to 
make an analogy between this sensitivity 
calculation and weighted ensemble, 
we should have $\nu(u)_p = \sum_{i:\xi_t^i \in p}\omega_t^i /\omega_t(u)$. Or in other 
words, we should consider perturbations 
that correspond to sampling from the bins 
according to particle weights, 
in accordance with Algorithm~\ref{alg1}.
But putting $\nu(u)_p = \sum_{i:\xi_t^i \in p}\omega_t^i /\omega_t(u)$ in~\eqref{Ntu2} gives 
something quite different from~\eqref{Ntu}. 
So while the sensitivity minimization formula~\eqref{Ntu2} is qualitatively similar to our mutation 
variance minimization formula~\eqref{Ntu}, the 
two are not actually the same.

\subsection{Proof of Theorem~\ref{thm_pert}}
\label{sec:appendix_proof}

We begin by noting the following.
\begin{equation}\label{b_poisson} 
\text{If }v^T \mathbbm{1} = 0, \quad \text{then}\quad g^T(I-K) = v^T, \,\,g^T \mathbbm{1} = 0\quad \Longleftrightarrow \quad g^T =\sum_{s=0}^\infty v^TK^s. 
\end{equation}
Like~\eqref{poiss}, this follows from ergodicity of $K$. See for instance~\cite{golub1986using}.

\begin{lemma}\label{lem1}
Suppose $A$ is a matrix with $A\mathbbm{1} = 0$. Then 
\begin{equation*}
\left.\frac{d}{d\epsilon}\mu(K + \epsilon A)\right|_{\epsilon = 0}f = 
 \mu^T A h.
\end{equation*}
\end{lemma}

\begin{proof}
Since $\mu(K+\epsilon A)^T(K+\epsilon A) = \mu(K+\epsilon A)^T$, we have 
\begin{align*}
0 = \left.\frac{d}{d\epsilon}\mu(K+\epsilon A)^T(I-K-\epsilon A)\right|_{\epsilon = 0} = 
\left.\frac{d}{d\epsilon}\mu(K+\epsilon A)^T\right|_{\epsilon = 0}(I-K) - \mu^T A.
\end{align*}
Thus 
\begin{equation}\label{above_1}
\left.\frac{d}{d\epsilon}\mu(K+\epsilon A)^T\right|_{\epsilon = 0}(I-K) =\mu^TA.
\end{equation}
Moreover since $\mu(Q)^T\mathbbm{1} = 1$ 
for any stochastic matrix $Q$, 
\begin{equation}\label{above_2}
0 = \left.\frac{d}{d\epsilon}\mu(K+\epsilon A)^T\mathbbm{1}\right|_{\epsilon = 0}\\
= \left.\frac{d}{d\epsilon}\mu(K+\epsilon A)^T\right|_{\epsilon = 0}\mathbbm{1}.
\end{equation}
Now by~\eqref{poiss},~\eqref{b_poisson},~\eqref{above_1}, and~\eqref{above_2}, 
\begin{align*}
\left.\frac{d}{d\epsilon}\mu(K+\epsilon A)^T\right|_{\epsilon = 0}f &= \sum_{s=0}^\infty \mu^TA K^sf \\
&= \mu^T A \sum_{s=0}^\infty K^s\bar{f} 
= \mu^T A h,
\end{align*}
where the last line above uses $AK^s \mathbbm{1} = A\mathbbm{1} = 0$ to replace $f$ with $\bar{f} = f - \mu^T f\mathbbm{1}$.
\end{proof}

\begin{lemma}\label{lem2}
Suppose $(A^{(n)})^{n=1,\ldots,N}$ are matrices such 
that {(i)} $(A^{(n)})^{n=1,\ldots,N}$ are independent over $n$,
{(ii)} $A^{(n)}\mathbbm{1} = 0$ for all $n$, 
and {(iii)} 
${\mathbb E}[A^{(n)}] = 0$ for all $n$.
Then 
\begin{align*}
{\mathbb E}\left[\left(\frac{d}{d\epsilon}\left.\mu\left(K + \epsilon \sum_{n=1}^N A^{(n)}\right)^T\right|_{\epsilon = 0}f\right)^2\right] =
\sum_{n=1}^N {\mathbb E}\left[\left(\frac{d}{d\epsilon}\left.\mu\left(K + \epsilon A^{(n)}\right)^T \right|_{\epsilon = 0}f\right)^2\right].
\end{align*}
\end{lemma}

\begin{proof}
Since $\mu$ is continuously differentiable, 
\begin{equation}\label{disp1}
\frac{d}{d\epsilon}\left.\mu\left(K + \epsilon \sum_{n=1}^N A^{(n)}\right)^T\right|_{\epsilon = 0}f
= \sum_{n=1}^N \frac{d}{d\epsilon}\left.\mu\left(K + \epsilon A^{(n)}\right)^T \right|_{\epsilon = 0}f.
\end{equation}
By  {\em (ii)-(iii)} and Lemma~\ref{lem1}, $
{\mathbb E}\left[\frac{d}{d\epsilon}\left.\mu\left(K + \epsilon A^{(n)}\right)^T \right|_{\epsilon = 0}f\right] = 0$ for all $n$.
So by {\em (i)},
\begin{equation}\label{disp2}
{\mathbb E}\left[\frac{d}{d\epsilon}\left.\mu\left(K + \epsilon A^{(n)}\right)^T \right|_{\epsilon = 0}f\frac{d}{d\epsilon}\left.\mu\left(K + \epsilon A^{(m)}\right)^T \right|_{\epsilon = 0}f\right] = 0, \qquad \text{if }n \ne m.
\end{equation}
The result follows by combining~\eqref{disp1} and~\eqref{disp2}.
\end{proof}

\begin{lemma}\label{lem3}
Let $A(u)$ be a random matrix with the distribution~\eqref{dist_A}.
Then
\begin{equation*}
{\mathbb E}\left[\left(\left.\frac{d}{d\epsilon}{\mu}(K + \epsilon A(u))^T\right|_{\epsilon = 0}f\right)^2\right] = \sum_{p \in u} \nu(u)_p\mu_p^2\left[(Kh^2)_p - (Kh)_p^2\right].
\end{equation*}
\end{lemma}

\begin{proof}
Note that $A(u)\mathbbm{1} = (  e_p e_q^T-e_pe_p^T K)\mathbbm{1} = 0$ since $K\mathbbm{1} = \mathbbm{1}$.
So by Lemma~\ref{lem1},
\begin{align*}
\left.\frac{d}{d\epsilon}\mu\left(K+\epsilon A(u)\right)\right|_{\epsilon = 0}^T f = \mu^T A(u) h.
\end{align*}
From this and~\eqref{dist_A},
\begin{align*}
{\mathbb E}\left[\left(\frac{d}{d\epsilon}\left.{\mu}(K + \epsilon A(u))^T\right|_{\epsilon = 0}f\right)^2\right] &= \sum_{p \in u} \sum_{q \in {\mathcal M}{\mathcal B}}\nu(u)_pK_{pq}\left[\mu^T(e_p e_q^T-e_p e_p^T K)h\right]^2 \\
&= \sum_{p \in u}\nu(u)_p\mu_p^2\sum_{q \in {\mathcal M}{\mathcal B}}K_{pq}\left[(e_q^T-e_p^T K)h\right]^2 \\
&= \sum_{p \in u}\nu(u)_p\mu_p^2\sum_{q \in {\mathcal M}{\mathcal B}} K_{pq} \left[(Kh)_p^2 + h_q^2 - 2(Kh)_p h_q\right] \\
&= \sum_{p \in u}\nu(u)_p\mu_p^2\left[(Kh^2)_p - (Kh)_p^2\right].
\end{align*}
\end{proof}

We are now ready to prove Theorem~\ref{thm_pert}.

\begin{proof}[Proof of Theorem~\ref{thm_pert}]
 By Lemmas~\ref{lem2} and~\ref{lem3},
\begin{align}\begin{split}\label{lim}
&{\mathbb E}\left[\left(\frac{d}{d\epsilon}\left.\mu\left(K + \epsilon B^{(N)}\right)^T \right|_{\epsilon = 0}f\right)^2\right] \\
&= N\sum_{u \in {\mathcal B}}\frac{1}{\lfloor N\lambda(u)\rfloor^2}\sum_{n=1}^{\lfloor N \lambda(u)\rfloor}{\mathbb E}\left[\left(\frac{d}{d\epsilon}\left.\mu\left(K + \epsilon A(u)^{(n)}\right)^T \right|_{\epsilon = 0}f\right)^2 \right] \\
&= \sum_{u \in {\mathcal B}}\frac{N}{\lfloor N\lambda(u)\rfloor}\sum_{p \in u}\nu(u)_p\mu_p^2\left[(Kh^2)_p - (Kh)_p^2\right]. \end{split}
\end{align}
The result follows by letting $N \to \infty$.
\end{proof}

\end{document}